\documentclass[10pt]{amsart}
\usepackage{a4wide}
\usepackage[utf8x]{inputenc}
\usepackage[colorlinks=true, pdfstartview=FitV, linkcolor=purple, citecolor=purple]{hyperref}
\usepackage{tikz,tikz-cd}
\usepackage{amsmath}
\usepackage{amsfonts}
\usepackage{amsthm}
\usepackage{amssymb}
\usepackage{lineno}
\usepackage{cleveref}
\usepackage{bbm}
\usepackage{tikz}
\usetikzlibrary{shapes, positioning, arrows.meta, calc,cd}
\tikzcdset{every label/.append style = {font = \normalsize}}
\usepackage{bbold}

\usepackage{tabularx}
\usepackage{float}
\usepackage{import}
\usepackage{enumitem}
\usepackage[toc,page]{appendix}
\usepackage{listings}
\usepackage{xcolor} 
\usepackage{comment}

\usepackage{todonotes}

\newtheorem{theorem}{Theorem}[section]

\newtheorem{lemma}[theorem]{Lemma}
\newtheorem{proposition}[theorem]{Proposition}

\theoremstyle{remark}
\newtheorem{definition}[theorem]{Definition}
\newtheorem{remark}[theorem]{Remark}

\numberwithin{equation}{section}

\newcommand{\cA}{\mathcal{A}}

\newcommand{\N}{\mathbb{N}}

\newcommand{\ba}{\mathbf{a}}

\newcommand{\bu}{\mathbf{u}}
\newcommand{\bv}{\mathbf{v}}
\newcommand{\bw}{\mathbf{w}}
\newcommand{\bx}{\mathbf{x}}
\newcommand{\by}{\mathbf{y}}
\newcommand{\bz}{\mathbf{z}}

\newcommand{\bone}{\mathbf{1}}

\newcommand{\btau}{\boldsymbol{\tau}}
\newcommand{\cL}{\mathcal{L}} 
\numberwithin{equation}{section}


\title{Convergence and combinatorics of the Reverse algorithm}

\author{Hiroaki Ito}
\author{Niels Langeveld}
\author{J\"org Thuswaldner}
\dedicatory{Dedicated to Wolfgang Steiner on the occasion of his 50$^{\,th}$ birthday}
\address{Chair of Mathematics, Statistics, and Geometry, Montanuniversit\"at Leoben, Franz-Josef-Strasse~18, A-8700 Leoben, Austria}
\date{\today}
\thanks{
This work was supported by the bilateral grant SYMDYNAR (ANR-23-CE40-0024 and FWF~I~6750) of the Agence Nationale de la Recherche and the Austrian Science Fund.}

\begin{document}
\begin{abstract}
We study the Reverse algorithm, a multidimensional continued fraction algorithm, which is not unimodular. We show that the Reverse algorithm is ergodic and, by proving that its second Lyapunov exponent is negative, that it is a.e.\ exponentially convergent. In addition to that, we attach substitutions to this algorithm and study the $S$-adic languages generated by sequences of these substitutions. The negativity of the second Lyapunov exponent implies that almost all of these languages are balanced. By a thorough study of the combinatorics of the substitutions, we are even able to obtain a concrete generic family of balanced languages that is characterized in terms of a simple condition on the underlying sequence of substitutions.
\end{abstract}
\maketitle
\section{Introduction}
In 1991, Arnoux and Rauzy \cite{Arnoux-Rauzy:91} invented a new 2-dimensional continued fraction algorithm that is defined on a subset of the positive cone $\Lambda=\{[x_0:x_1:x_2]\in\mathbb{P}^2\,:\, x_0, x_1, x_2>0\}$ of the projective space $\mathbb{P}^2$. Given an element $[x_0:x_1:x_2]\in\Lambda$, the continued fraction map $F_{\mathrm{AR}}$ defining this algorithm subtracts the two smaller coordinates from the largest coordinate of its argument. For example, if $[x_0:x_1:x_2]\in\Lambda$ with $x_0>x_1$ and $x_0 > x_2$ is given, then $F_{\mathrm{AR}}([x_0:x_1:x_2])=[x_0-x_1-x_2:x_1:x_2]$. 
It is easy to see that this procedure does not always result in a vector that is contained in~$\Lambda$. Therefore, in general, one cannot iterate the map $F_{\mathrm{AR}}$ on~$\Lambda$. To make sure that all iterations of $F_{\mathrm{AR}}$ remain in~$\Lambda$, one needs to consider the restriction $F_{\mathrm{AR}}:\mathcal{R}\to\mathcal{R}$, where $\mathcal{R}$ is the well-known {\em Rauzy gasket}, a ``fractal'' set of measure~$0$, which has been extensively studied in the literature (see {\em e.g.}~\cite{Arnoux_2013,AHS:16,PS:24}). It is known that the {\em Arnoux-Rauzy algorithm} $(\mathcal{R}, F_{\mathrm{AR}})$ is exponentially convergent in the sense that it has a second Lyapunov exponent which is negative~\cite{AD:19}. We can attach {\em Arnoux-Rauzy substitutions} with the Arnoux-Rauzy algorithm. Negativity of the second Lyapunov exponent entails that these substitutions have nice properties. For example, in \cite{Berthe-Cassaigne-Steiner, Delecroix-Hejda-Steiner} strong balancedness properties of $S$-adic languages, that are defined in terms of sequences of Arnoux-Rauzy substitutions, are established. More recently, in \cite{BST:19}, it was shown that $S$-adic dynamical systems of Arnoux-Rauzy substitutions are often measurably conjugate to translations on the $2$-dimensional torus.

The problem with the domain of definition of $F_{\mathrm{AR}}$ comes from the fact that  elements contained in the set 
$
\Lambda_4=\{[x_0:x_1:x_2]\in\Lambda : 2x_i \le x_0 + x_1 + x_2 \text{ for all } i\in \{0,1,2\} \}
$
are mapped outside of $\Lambda$ by $F_{\mathrm{AR}}$. We can plug this hole by changing the definition of the map $F_{\mathrm{AR}}$ on $\Lambda_4$. This leads to the {\em Reverse algorithm}, a $2$-dimensional continued fraction algorithm that was introduced in \cite{AL18}.

\begin{definition}[Reverse algorithm; see {\cite[Section~4]{AL18}}]
\label{def:Reverse}
Let $\Lambda$ be the positive cone in $\mathbb{P}^2$, which is partitioned by the subcones 
\begin{align*}
\Lambda_1&=\{[x_0:x_1:x_2]\in\Lambda \,:\, 2x_0>x_0+x_1+x_2\},\\
\Lambda_2&=\{[x_0:x_1:x_2]\in\Lambda \,:\, 2x_1>x_0+x_1+x_2\},\\
\Lambda_3&=\{[x_0:x_1:x_2]\in\Lambda \,:\, 2x_2>x_0+x_1+x_2\},\\
\Lambda_4&=\{[x_0:x_1:x_2]\in\Lambda \,:\, 2x_i \le x_0 + x_1 + x_2 \, \text{ for all } i\in \{0,1,2\} \}.
\end{align*}
The {\em Reverse algorithm} $(\Lambda, F_{\mathrm{R}})$ is defined by the map $F_{\mathrm{R}}:\Lambda \to \Lambda$, 
\begin{align}\label{eq:Frdef}
F_{\mathrm{R}}([x_0:x_1:x_2])
=\begin{cases}
[x_0-x_1-x_2:x_1:x_2] & \text{if } [x_0:x_1:x_2]\in\Lambda_1, \\
[x_0:-x_0+x_1-x_2:x_2] & \text{if } [x_0:x_1:x_2]\in\Lambda_2, \\
[x_0:x_1:-x_0-x_1+x_2] & \text{if } [x_0:x_1:x_2]\in\Lambda_3, \\
[-x_0+x_1+x_2:x_0-x_1+x_2:x_0+x_1-x_2] & \text{if } [x_0:x_1:x_2]\in\Lambda_4.
\end{cases} 
\end{align}
\end{definition}
It is easy to check that\footnote{Here and in the sequel, we often ignore sets of measure zero.} $F_{\mathrm{R}}(\Lambda_i)=\Lambda$ holds for $i\in\{1,2,3,4\}$. 

In the above definition, we defined the Reverse algorithm in a subset of $\mathbb{P}^2$ in order to emphasize on its relation to the Arnoux-Rauzy algorithm $(\mathcal{R}, F_{\mathrm{AR}})$, and because the formulas are easier in this projective setting. Later, in \eqref{eq:smallFrdef}, we will define a version of $F_\mathrm{R}$ that lives in a subset of $\mathbb{R}^2$.

\begin{remark}\label{rem:sorted}
One can also define a sorted version of the Reverse algorithm. This sorted algorithm is defined on $\Lambda'=\{[x_0:x_1:x_2]\in \mathbb{P}^2\,:\, x_0>x_1>x_2>0\}$ by $\mathop{sort}\circ F_{\mathrm{R}}$, where $\mathop{sort}:\mathbb{P}^2 \to \mathbb{P}^2$ orders the coordinates of a vector descendingly. In the present paper, we mainly deal with the unsorted version of the Reverse algorithm from Definition~\ref{def:Reverse}. The sorted version is briefly discussed in \Cref{sec:appA}.
\end{remark}

Since the Reverse algorithm extends the Arnoux-Rauzy algorithm to a set of full measure, it is desirable to prove the above-mentioned results on the Arnoux-Rauzy algorithm also for the Reverse algorithm. Although results of this kind are already known for algorithms that are defined on a set of full measure, like, for instance, for the Brun algorithm \cite{Delecroix-Hejda-Steiner,BST:23,ABMST25} or the Cassaigne-Selmer algorithm~\cite{CCL22}, the Reverse algorithm deserves particular interest. Firstly, contrary to the algorithms mentioned so far, the Reverse algorithm is not {\em unimodular}, in particular, its action on $\Lambda_4$ corresponds to the nonunimodular matrix $M_4$ defined in \eqref{eq:matricesM} below. For nonunimodular algorithms the theory is significantly less developed and $S$-adic languages or dynamical systems attached to such algorithms are much less studied (see \cite{Langeveld-Rossi-Thuswaldner} for a study of the $1$-dimensional $N$-continued fraction algorithm, which is nonunimodular for $N>1$). Secondly, natural generalizations of the Arnoux-Rauzy algorithm to higher dimensions are strongly, and even exponentially, convergent (see~\cite{AD:19}). This makes the Reverse algorithm a candidate for an algorithm that is defined on a set of positive measure and that remains strongly convergent also when suitably generalized to higher dimensions. This is remarkable, because according to numerical experiments performed in~\cite{BST21}, the known classical algorithms seem to cease to be strongly convergent in high dimensions. So far, an algorithm that is defined on a set of positive measure being strongly convergent in all dimensions seems to be known only if we allow for negative entries in the matrices defining the algorithm (we refer to~\cite{Lagarias:94} for an example of such a ``nonpositive'' algorithm). 

In the present paper, we concentrate on the $2$-dimensional Reverse algorithm from Definition~\ref{def:Reverse}. Let 
\[
\Delta=\{\bx\in\mathbb {R}^3_{>0} \,:\, \|\bx\|_1=1\},
\]
and define the isomorphism (the superscript ``$t$'' in front of a vector denotes transposition)
\[
\kappa: \Lambda \to \Delta; \quad [x_0:x_1:x_2] \mapsto \,^{\displaystyle ^t}\!\bigg(\frac{x_0}{x_0+x_1+x_2},\frac{x_1}{x_0+x_1+x_2},\frac{x_2}{x_0+x_1+x_2}\bigg).
\]
In the sequel, it will be more convenient to work with the {\em projectivized version} $(\Delta,f_{\mathrm{R}})$ of the Reverse algorithm $(\Lambda,F_{\mathrm{R}})$ that is defined by the requirement that the diagram
\begin{equation}\label{commDiagfF}
\begin{tikzcd}
\Lambda \arrow{r}{F_\mathrm{R}} \arrow{d}[swap]{\kappa} & \Lambda \arrow{d}{\kappa} \\
\Delta \arrow{r}[swap]{f_\mathrm{R}}  & \Delta 
\end{tikzcd}
\end{equation}
is commutative. Thus $(\Delta,f_{\mathrm{R}})$ and $(\Lambda, F_{\mathrm{R}})$ are isomorphic versions of the Reverse algorithm. For later reference, we want to give an explicit definition of $f_{\mathrm{R}}$. Set $\Delta(i) := \kappa(\Lambda_i)$, $i\in\{1,2,3,4\}$, to obtain the partition $\{\Delta(1),\Delta(2),\Delta(3),\Delta(4)\}$ of $\Delta$. Let $\mathcal{M}(3,\mathbb{N})$ be the set of regular $3\times 3$ nonnegative integer matrices and set
\begin{align}\label{eq:coReverse}
A:\Delta \to \mathcal{M}(3,\mathbb{N});\quad  \bx \mapsto \,^t\!M_i  \quad \text{if and only if} \quad \bx \in\Delta(i) \qquad(i\in\{1,2,3,4\}),
\end{align}
where
\begin{align}\label{eq:matricesM}
M_1=
\begin{pmatrix}
1 & 1 & 1\\
0 & 1 & 0\\
0 & 0 & 1
\end{pmatrix}, \quad 
M_2=
\begin{pmatrix}
1 & 0 & 0\\
1 & 1 & 1\\
0 & 0 & 1
\end{pmatrix}, \quad 
M_3=
\begin{pmatrix}
1 & 0 & 0\\
0 & 1 & 0\\
1 & 1 & 1
\end{pmatrix}, \quad 
M_4=
\begin{pmatrix}
0 & 1 & 1\\
1 & 0 & 1\\
1 & 1 & 0
\end{pmatrix}.
\end{align}
Then the projectivized version of the Reverse algorithm $(\Delta,f_\mathrm{R})$ is given by the map
\begin{equation}\label{eq:frexpl}
 f_{\mathrm{R}}: \Delta\to\Delta;\qquad \bx \mapsto \frac{\,^t\!A(\bx)^{-1}\bx}{\|\,^t\!A(\bx)^{-1}\bx\|_1}.
\end{equation}
The transposition in the definition of $A$ in \eqref{eq:coReverse} is needed because it entails that $A$ forms a {\em cocycle} in a sense that is defined in Section~\ref{sec:MUl} below.
Using the matrices in \eqref{eq:matricesM}, we can write out the map $f_{\mathrm{R}}$ given in \eqref{eq:frexpl} as
\begin{align}\label{eq:smallFrdef}
f_{\mathrm{R}}: 
\begin{pmatrix}x_0 \\ x_1\\ x_2\end{pmatrix} \mapsto
\begin{cases}
\,^{\displaystyle ^t}\!\!\big(\frac{x_0-x_1-x_2}{x_0}, \frac{x_1}{x_0}, \frac{x_2}{x_0}\big)  & \text{if }\bx\in\Delta(1), \\[2pt]
\,^{\displaystyle ^t}\!\!\big(\frac{x_0}{x_1}, \frac{-x_0+x_1-x_2}{x_1}, \frac{x_2}{x_1}\big)  & \text{if }\bx\in\Delta(2), \\[2pt]
\,^{\displaystyle ^t}\!\!\big(\frac{x_0}{x_2}, \frac{x_1}{x_2}, \frac{-x_0-x_1+x_2}{x_2}\big)  & \text{if }\bx\in\Delta(3), \\[2pt]
\,^{\displaystyle ^t}\!\!\big(\frac{-x_0+x_1+x_2}{x_0+x_1+x_2}, \frac{x_0-x_1+x_2}{x_0+x_1+x_2}, \frac{x_0+x_1-x_2}{x_0+x_1+x_2}\big) & \text{if }\bx\in\Delta(4).
\end{cases}
\end{align}
By comparing \eqref{eq:Frdef} and \eqref{eq:smallFrdef}, it is easy to see that the function $f_{\mathrm{R}}$  in \eqref{eq:smallFrdef} is the same as the one satisfying \eqref{commDiagfF}.

In \cite{AL18}, the authors construct a natural extension of the Reverse algorithm. Using this natural extension they show that the absolutely continuous measure $\mu$ with density
\begin{align}\label{eq:density}
h(x_0,x_1,x_2)=\frac{4}{\pi^2(1-x_0)(1-x_1)(1-x_2)}
\end{align}
is an invariant probability measure for the Reverse algorithm $(\Delta, f_\mathrm{R})$.
The density $h$ has three singular points $(1,0,0), (0,1,0), (0,0,1)$, nevertheless, the total mass of $\mu$ is finite. Indeed, 
\begin{align*}
\mu(\Delta)=&\frac4{\pi^2}\int_{0}^{1}\int_{0}^{1-x}\frac{1}{(1-x)(1-y)(x+y)}dydx\\
=&\frac4{\pi^2}\int_{0}^{1}\frac{1}{(1-x)}\int_{0}^{1-x}\frac{1}{(x+1)}\left(\frac{1}{1-y}+\frac{1}{x+y}\right)dydx\\
=&\frac4{\pi^2}\int_{0}^{1}\frac{2\log{x}}{(x+1)(x-1)}dx
=-\frac8{\pi^2}\int_{0}^{1}\sum_{n=0}^{\infty}x^{2n}\log{x}dx\\
=&\frac8{\pi^2}\sum_{n=0}^{\infty}\frac{1}{(2n+1)^2}
=1
\end{align*}
(see also \cite[Remark~18]{AL18}).

We want to attach substitutions and ``directive sequences'' of substitutions to the Reverse algorithm. Before we do this, we provide some definitions. 

For a finite set $\cA$, called the {\em alphabet},
let $\cA^*$ be the set of finite words over $\cA$. Then $\cA^*$ is a free monoid w.r.t.\ the operation of concatenation.
Moreover, let $\cA^\N$ be the set of one-sided infinite words over $\cA$. For a word $w=w_0\cdots w_{n-1}\in\cA^*$, let $|w|=n$ be the {\em length} of $w$, {\em i.e.}, the number of letters in $w$. The only word of length $0$ is called the {\em empty word} and will be denoted by $\epsilon$.
For $v=v_0\cdots v_{m-1}\in\cA^*\setminus\{\epsilon\}$ let $|w|_v$ denote the number of (possibly overlapping) occurrences of $v$ in $w$, {\it i.e.},
\[
|w|_v = \#\{
i \in \{0,\ldots, n-m\} \;:\;
w_i\cdots w_{i+m-1} = v_0\cdots v_{m-1}
\}
\]
(note that $|w|_v=0$ if $m > n$). The {\em abelianization function} $\mathbf{l}:\cA^*\to\N^{\#\cA}$ is given by $v \mapsto (|v|_i)_{i\in \cA}$. A {\em factor} of a finite or infinite word $w=w_0w_1\cdots \in \cA^*\cup \cA^\N$ is a finite word of the form $v=w_k\cdots w_{\ell}$ ($k\le \ell$). We write $v\prec w$, if $v$ is a factor of $w$. If $k=0$, then the factor $v$ is called a {\em prefix} of $w$. 

Fix $C>0$. A pair $(v_1,v_2)\in\cA^*\times \cA^*$ of words\footnote{Note that $v_1$ and $v_2$ are not assumed to have the same length here.} is called {\em $C$-letter balanced}, if for each $i\in\cA$ we have $\big| |v_1|_i-|v_2|_i\big| \le C$. A finite or infinite word $w \in \cA^*\cup \cA^\N$ is called {\em $C$-letter balanced}, if every pair $(v_1,v_2)$ of factors of $w$ satisfying $|v_1|=|v_2|$ is $C$-letter balanced. In each case, {\em letter balanced} means $C$-letter balanced for some $C>0$. A language $\cL\subset \mathcal{A}^*$ is called {\em $C$-factor balanced}, if
\begin{align*}
\big| |w|_v-|w'|_v \big| \le C \quad \text{for all $w, w'\in\cL$ with $|w|=|w'|$}
\end{align*}
holds for each $v\in\mathcal{A}^*\setminus\{\epsilon\}$.
It is called {\em factor balanced}, if it is $C$-factor balanced for some $C>0$.

A {\em substitution} $\sigma$ is an endomorphism of the free monoid $\mathcal{A}^*$ ({\em i.e.}, $\sigma(vw)=\sigma(v)\sigma(w)$ for all $v,w \in \mathcal{A}^*$) that is {\em nonerasing} in the sense that $|\sigma(i)|\ge 1$ holds for each $i\in\mathcal{A}$.
A substitution over the alphabet $\mathcal{A}$ is {\em left} (resp.\ \em{right}) \em{proper}, if the image of each letter starts (resp.\ ends) with the same letter.  If $\sigma$ is a substitution over the alphabet $\cA$, the {\em incidence matrix }of $\sigma$ is given by the $|\cA|\times|\cA|$-matrix $B_\sigma=(|\sigma(j)|_i)_{1\le i,j\le |\cA|}$. It is immediate from the definitions that $\mathbf{l}\circ\sigma=B_\sigma\circ \mathbf{l}$. In this sense, $B_\sigma$ can be regarded as the abelianization of $\sigma$.

Let $S$ be a set of substitutions over the same alphabet $\mathcal{A}$. A {\em directive sequence} over $S$ is a sequence $\btau=(\tau_n)_{n\in \N}$. Since we will often need compositions of consecutive substitutions of a directive sequence, we use the notation
\[ 
\tau_{[k,\ell)} = \tau_k\circ\cdots \circ \tau_{\ell-1} \qquad (k\le \ell),
\]
where $\tau_{[k,k)}$ is the {\em identity substitution} defined by $i\mapsto i$ ($i\in \cA$).
The {\em $S$-adic language} of a directive sequence $\btau=(\tau_n)_{n\in\N}$ is defined as 
\[
\cL_{\btau} = \{v \prec \tau_{[0,n)}(i) \;:\; n\in \N, \, i\in \cA \}.
\]
Let $\Sigma((\tau_n)_{n\in\N})=(\tau_{n+1})_{n\in\N}$ be the left shift on the set of directive sequences.
We also use the languages $\cL_{\btau}^{(n)}=\cL_{\Sigma^n\btau}$ of the shifted directive sequences $\Sigma^n\btau$.

We study the substitutions over the alphabet $\cA=\{1,2,3\}$ given by
\begin{equation}\label{eq:reversSubst} 
\sigma_1:
			\begin{cases}
				1\mapsto 1,\\
				2\mapsto 21,\\
				3\mapsto 31,\\
			\end{cases}  \qquad 
\sigma_2:
			\begin{cases}
				1\mapsto 12,\\
				2\mapsto 2,\\
				3\mapsto 32,\\
			\end{cases}  \qquad 
\sigma_3:
			\begin{cases}
				1\mapsto 13,\\
				2\mapsto 23,\\
				3\mapsto 3,\\
			\end{cases}  \qquad 
\sigma_4:
			\begin{cases}
				1\mapsto 23,\\
				2\mapsto 31,\\
				3\mapsto 12.\\
			\end{cases}   
\end{equation}
Note that $B_{\sigma_i}=M_i$, {\it i.e.}, the incidence matrices of $\sigma_i$ is given by the matrix $M_i$ from \eqref{eq:matricesM} ($i\in\{1,2,3,4\}$). Note that, like for the well-known {\em Arnoux-Rauzy substitutions}~$\sigma_1$, $\sigma_2$, and~$\sigma_3$, also the choice of $\sigma_4$ is not ``canonical''. One could redefine each of these substitutions by permuting the images of each letter in an arbitrary way, without changing the incidence matrix. Our definition of~$\sigma_4$ is the most ``symmetric'' version. 
We expect that redefining the substitutions in a less symmetric way leads to worse balance properties, however, we will not go into this issue in the present paper.  

When iterating the map $f_\mathrm{R}$ of the Reverse algorithm, the matrix valued map $A$ produces an infinite sequence of matrices $(M_{i_n})_{n\in\N}\in \{M_1,M_2,M_3,M_4\}^\N$. Because $M_i$ is the incidence matrix of the substitution $\sigma_i$, we can attach to the sequence $(M_{i_n})_{n\in\N}$ the directive sequence $(\sigma_{i_n})_{n\in\N}\in \{\sigma_1,\sigma_2,\sigma_3,\sigma_4\}^\N$. Thus the Reverse algorithm can be seen as a device that generates infinite sequences of substitutions. 
To be more precise, in Lemma~\ref{lem:topconv} we will show that the Reverse algorithm is a.e.\ topologically convergent in the sense that we can attach to a.e.\ $\bx\in\Delta$ a sequence from $\{\sigma_1,\sigma_2,\sigma_3,\sigma_4\}^\N$ in a one-to-one way. Indeed, there is a continuous, surjective, a.e.\ bijective function $\Phi:\{\sigma_1,\sigma_2,\sigma_3,\sigma_4\}^\N\to\Delta$ such that the diagram 
\[
\begin{tikzcd}
\{\sigma_1,\sigma_2,\sigma_3,\sigma_4\}^\N\arrow[r, "\Sigma"]\arrow[d,"\Phi"] &\{\sigma_1,\sigma_2,\sigma_3,\sigma_4\}^\N \arrow[d, "\Phi"] \\
\Delta \arrow[r, "f_{\mathrm{R}}"]& \Delta
\end{tikzcd}
\]
commutes. Thus, we have the measurable conjugacy $(\Delta,f_{\mathrm{R}},\mu) \cong (\{\sigma_1,\sigma_2,\sigma_3,\sigma_4\}^\N, \Sigma, \Phi^*\mu)$, where $\Phi^*\mu=\mu\circ\Phi$ is the pullback measure of $\mu$. Moreover, $(\Delta,f_{\mathrm{R}})$ is a topological factor of $(\{\sigma_1,\sigma_2,\sigma_3,\sigma_4\}^\N, \Sigma)$. 
In a certain sense, the dynamical system on the directive sequences forms a ``nonabelian'' version of the Reverse algorithm. It turns out that the algorithm can be used to derive properties of these sequences of substitutions and vice versa (see for instance \cite{BST:23,ABMST25} for more details on this).

The aim of the present paper is to prove the following results. Our first result, which is proved in Section~\ref{sec:erg}, is on ergodicity of the Reverse algorithm.

\begin{theorem}\label{ergodicity}
The Reverse algorithm $(\Delta, f_{\mathrm{R}})$ is ergodic with respect to the invariant measure $\mu$ with density $h$ defined in \eqref{eq:density}.
\end{theorem}

Our second result contains an estimate of the second Lyapunov exponent of the Reverse algorithm. Labb\'e~\cite{Lab15} did some computer experiments that indicate that the second Lyapunov exponent of the Reverse algorithm is about $-0.103\ldots$ However, these experiments do not constitute a rigorous proof of the negativity of the second Lyapunov exponent. We are able to give the following result, which will be proved in Section~\ref{sec:lyap}.

\begin{theorem}\label{second Lyapunov negative}
The second Lyapunov exponent for the Reverse algorithm is negative. In particular, it is less than $-0.020608$. Thus the Reverse algorithm is a.e.\ exponentially convergent.
\end{theorem}

We note that a negative second Lyapunov exponent implies a.e.\ letter balancedness of the $S$-adic languages $\cL_{\btau} $ (see e.g.~\cite[Theorem 13.6]{ABMST25} which contains the unimodular version of this assertion; the nonunimodular case follows along the same lines). However, this does not exhibit any concrete directive sequence $\btau$ for which $\cL_{\btau}$ is letter balanced. Section~\ref{sec:bal} is devoted to the proof of the following theorem which exhibits a concrete family of languages that are even factor balanced. 
In its statement we say that a word $\zeta\in \{\sigma_1,\sigma_2,\sigma_3,\sigma_4\}^*$ has positive density in $(\tau_n)_{n\in\N} \in \{\sigma_1,\sigma_2,\sigma_3,\sigma_4\}^\N$, if\footnote{Note that $\tau_0\cdots\tau_{n-1} \in \{\sigma_1,\sigma_2,\sigma_3,\sigma_4\}^*$ is a word over the alphabet $\{\sigma_1,\sigma_2,\sigma_3,\sigma_4\}$ and not a composition of substitutions.} 
$\lim_{n\to \infty} \frac{|\tau_0\cdots\tau_{n-1}|_{\zeta}}n >0$.

\begin{theorem}\label{th:balance}
If $\btau \in \{\sigma_1,\sigma_2,\sigma_3,\sigma_4\}^\N$ is a directive sequence in which the word $({\sigma_1}{\sigma_2}{\sigma_3})^{9}$ has positive density, then $\cL_{\btau}$ is factor balanced.
\end{theorem}

It follows from the ergodic theorem that the family of directed sequences covered by \Cref{th:balance} has full measure w.r.t.\ any ergodic measure on the full shift $(\{\sigma_1,\sigma_2,\sigma_3,\sigma_4\}^\N,\Sigma)$, which assigns positive mass to the cylinder $({\sigma_1}{\sigma_2}{\sigma_3})^{9}$. Examples for such measures are the Bernoulli measure and the pullback measure $\Phi^*\mu$ of the measure $\mu$ (these two measures are mutually singular as can be seen by arguing as in \cite[Remark~6.2]{CCL22}). We also mention that Theorem~\ref{th:balance} is {\em a fortiori} true if $\sigma_4$ does not occur in $\btau$. Thus it contains the analogous result for the Arnoux-Rauzy substitutions $\sigma_1,\sigma_2,\sigma_3$ ({\em cf.}~\cite[Theorem~1]{Delecroix-Hejda-Steiner}). 
Moreover, by the same proof one can show variants of Theorem~\ref{th:balance} with $\sigma_1\sigma_2\sigma_3$ replaced by other blocks of substitutions with positive incidence matrix.

In \cite[Theorem~C]{CCL22} analogs of Theorem~\ref{second Lyapunov negative} and~\ref{th:balance} are proved for directive sequences of the Cassaigne algorithm. As in \cite{CCL22} and many other papers, our proof of the negativity of the second Lyapunov exponent is computational and uses ideas going back to Lagarias~\cite{Lagarias:93} (see, for instance, \cite{Hardcastle:02, Delecroix-Hejda-Steiner,BST21} for other proofs of that flavor). However, as in \cite{CCL22}, we use the $1$-norm rather than the $\infty$-norm in our argument. As for the balance result, like in the proof of \cite[Theorem~C]{CCL22}, one of the main ideas of our proof is based on some ``contractive cylinder'' (see Lemma~\ref{lem:DHSLem6}). However, since the proof of balance results strongly depends on the combinatorial properties of the involved substitutions, our proof considerable differs from the one of \cite[Theorem~C]{CCL22}. We mention that the property that we establish in  Proposition~\ref{lem:sigma_ibalance} is of crucial importance in our arguments. Therefore, we think that our proof has the potential to be generalized to substitutions having a similar property.

\section{$2$-dimensional continued fraction algorithms}\label{sec:MUl}
In this section, we provide the necessary results from the theory of multidimensional continued fraction algorithms. For the general theory of multidimensional continued fraction algorithms, we refer {\em e.g.}\ to~\cite{Schweiger:00, AL18, BST21, BST:23}. Our definition of a multidimensional continued fraction algorithm is an extension of \cite[Section~9.1]{ABMST25} to the nonunimodular situation. Since we are only interested in $2$-dimensional algorithms, we confine ourselves to this case.

\begin{definition}[$2$-dimensional continued fraction algorithm] \label{def:cf} 
Let $\Delta=\{\bx\in\mathbb {R}^3_{>0} \,:\, \|\bx\|_1=1\}$ and 
\begin{equation}\label{eq:coAllg}
A:\, \Delta \to \mathcal{M}(3,\mathbb{N})
\end{equation}
be a map from $\Delta$ to the set $\mathcal{M}(3,\mathbb{Z})$ of regular $3 {\times} 3$ nonnegative integer  matrices such that $\frac{^t{\!A}(\bx)^{-1}\bx}{\|^t{\!A}(\bx)^{-1}\bx\|_1} \in \Delta$ for all $\bx \in \Delta$.
Then the pair $(\Delta,f)$ with
\begin{equation}\label{eq:CFalgo}
f:\, \Delta \to \Delta; \quad \bx \mapsto \frac{^t{\!A}(\bx)^{-1} \bx}{\|^t{\!A}(\bx)^{-1}\bx\|_1}
\end{equation}
is called a \emph{$2$-dimensional continued fraction algorithm}.
\end{definition}

We assume throughout the paper that $(\Delta,f)$ admits an invariant measure $\mu$ that is equivalent to the Lebesgue measure on $\Delta$.

If we iterate an algorithm $(\Delta,f)$ on a point $\bx \in \Delta$, the mapping $A$ from \eqref{eq:coAllg} yields a sequence of matrices 
$(M_{a_i})_{i\in \N}$ from $\mathcal{M}(3,\mathbb{N})$. We will study such sequences when investigating the Reverse algorithm, and the following properties will be of relevance. We call a sequence $(M_{a_i})_{i\in \N}$ {\em primitive}, if for each $m\in \N$ there is $n\in \N$ with $n>m$ such that $M_{a_m} \cdots M_{a_{n-1}}$ is a positive matrix. Moreover, $(M_{a_i})_{i\in \N}$ is called {\em recurrent}, if for each $m\in \N$ there is $n\ge 1$ 
such that $(M_0,\ldots, M_{m-1}) = (M_{n},\ldots, M_{n+m-1})$. If there exists a vector $\bu\in \mathbb{R}^3_{\ge 0}$ satisfying 
\[
\bigcap_{n\ge 0}M_{a_0}\cdots M_{a_{n-1}}\mathbb{R}^d_{\ge 0}
= \mathbb{R}_{\ge 0}\bu,
\]
we say that $(M_{a_i})_{i\in \mathbb{N}}$ admits a {\em generalized right eigenvector} $\bu$. We need the following criterion.

\begin{lemma}[see {\cite[Proposition~3.5.5 and its proof]{thuswaldner2019boldsymbolsadic}}]\label{lem:3.5.5}
Let $(M_{a_i})_{i\in \mathbb{N}}$ be a sequence of matrices taken from $\mathcal{M}(3,\mathbb{N})$. If there is a positive matrix $M$ and $r\in\N$ such that $M=M_{a_n}\cdots M_{a_{n+r}}$ holds for infinitely many $n\in \N$, then $(M_{a_i})_{i\in \mathbb{N}}$ admits a generalized right eigenvector $\bu\in \mathbb{R}^d_{>0}$. 
\end{lemma}

It is easy to see that the  condition of the lemma certainly holds if $(M_{a_i})_{i\in \mathbb{N}}$
is primitive and recurrent.

The transposition used in \eqref{eq:CFalgo} comes from the fact that we are interested in the \emph{linear cocycle} of the algorithm $(\Delta,f)$, which is given by~$A$. Indeed,
$$
A^{(n)}(\bx) = A(f^{n-1}\bx) \cdots A(f \bx) \, A(\bx)
$$
satisfies the \emph{cocycle property} 
$$
A^{(m+n)}(\bx) = A^{(m)}(f^n\bx)\, A^{(n)}(\bx).
$$ 
For a linear cocycle, we can define Lyapunov exponents.  To this end, let $\delta_1(M)$, $\delta_2(M)$, $\delta_3(M)$ be the singular values of a $3\times 3$ matrix~$M$ ordered in a way that $\delta_1(M)\ge \delta_2(M)\ge \delta_3(M)$. Then  
\begin{align*}\label{Lyapnov exponent}
\lambda_i(A,\bx)=\lim_{n \rightarrow \infty}\frac{1}{n}\log{\delta_i(A^{n}(\bx))}
\end{align*}
is called the {\em $i$-th Lyapunov exponent for $A$ at $\bx$}, provided that the limit exists ($i\in\{1,2,3\}$).
By definition, $\lambda_1(A,\bx) \ge\lambda_2(A,\bx) \ge\lambda_3(A,\bx)$. If $(\Delta, f,\mu)$ is ergodic, these values are the same for $\mu$-a.e.\ $\bx \in\Delta$ and we call these common values $\lambda_1(A) \ge\lambda_2(A) \ge\lambda_3(A)$ the {\em Lyapunov exponents for~$A$}. In particular, in case of ergodicity we have for the first Lyapunov exponent that 
\begin{equation}\label{eq:lyapudef}
\begin{split}
\lambda_1(A)&= \lim_{n \rightarrow \infty}\frac{1}{n}\log{\delta_1(A^{n}(\bx_0))} 
\\
&=
\lim_{n \rightarrow \infty}\frac{1}{n}\int_{\Delta}\log{\delta_1(A^{n}(\bx))} \mathrm{d}\mu(\bx) =
\lim_{n \rightarrow \infty}\frac{1}{n}\int_{\Delta}\log{\|A^{n}(\bx)\|} \mathrm{d}\mu(\bx)
\end{split}
\end{equation}
for a.e.\ $\bx_0\in\Delta$. Note that the last integral does not depend on the norm.
For details on Lyapunov exponents we refer for instance to \cite{Arnold98}. Lyapunov exponents, in particular, the second Lyapunov exponent, are intimately related to the convergence behavior of $(\Delta, f)$, see {\em e.g.} \cite{Lagarias:93,BST21}.

A \textit{cylinder} of rank $n$ in $\Delta$ is a set
\begin{equation}\label{eq:DeltaCyl}
\Delta(a_0a_1\cdots a_{n-1}):=\{\bx\in \Delta \;:\; \bx\in \Delta(a_0),\, f(\bx)\in \Delta(a_1),\, \dots,\, f^{n-1}(\bx)\in \Delta(a_{n-1})\},
\end{equation}
where $a_0\cdots a_{n-1} \in \cA^*$.
The cylinder $\Delta(a_0 \cdots a_{n-1})$ containing $\bx$ has the row vectors of $A^{(n)}(\bx)$ as vertices. According to the definition of $A$ we have
$A^{(n)}(\Delta(a_0\cdots a_{n-1})) = \{ 
^t(M_{a_0}\cdots M_{a_{n-1}})\}$,
a singleton, which we will often identify with the single matrix $^t(M_{a_0}\cdots M_{a_{n-1}})$ it contains. 

Let us recall some further facts from the general theory of multidimensional continued fraction algorithms and formulate them for the $2$-dimensional case.

\begin{definition}[Topological convergence]\label{def:topConv}
Let $(\Delta,f)$ be a $2$-dimensional continued fraction algorithm and let $\bx=\,^t(x_0,x_1,x_2)\in\Delta$.
Choose $a_0a_1\cdots\in\cA^\N$ such that $\bx\in \Delta(a_0 a_1 \cdots a_{n-1})$ holds for each $n\ge 0$.
If the diameters $d(\Delta(a_0 a_1 \cdots a_{n-1}))$ of $\Delta(a_0a_1\cdots a_{n-1})$ satisfy
\begin{align*}
\lim_{n\rightarrow \infty}d(\Delta(a_0 a_1 \cdots a_{n-1}))=0,
\end{align*}
then we say that $(\Delta,f)$ is {\em topologically convergent} at $\bx$.
\end{definition}

It is easy to see that topological convergence at $\bx$ is equivalent to the fact that
$(M_{a_i})_{i\in \mathbb{N}}$ has a generalized right eigenvector. A much stronger notion of convergence is the following one (see also \cite{Lagarias:93}, where this has been introduced).

\begin{definition}[Exponential convergence]\label{def:expConv}
Let $(\Delta,f)$ be a $2$-dimensional continued fraction algorithm and let $\bx=\,^t(x_0,x_1,x_2)\in\Delta$ be arbitrary. For $n\in\N$, let the entries of the cocycle matrix be given by 
\begin{equation}\label{eq:Acomponents}
A^{(n)}(\bx) =
\begin{pmatrix}
p_{00}^{(n)} & p_{01}^{(n)} & p_{02}^{(n)}\\
p_{10}^{(n)} & p_{11}^{(n)} & p_{12}^{(n)}\\
p_{20}^{(n)}\textbf{} & p_{21}^{(n)} & p_{22}^{(n)}
\end{pmatrix}
\end{equation}
(we suppress the dependency of the entries on $\bx$ for notational convenience). For $i\in\{0,1,2\}$ and $n\in\N$ set $p_{i}^{(n)}=p_{i0}^{(n)}+p_{i1}^{(n)}+p_{i2}^{(n)}$.  We say that $\Big(\frac{p_{i1}^{(n)}}{p_{i}^{(n)}}, \frac{p_{i2}^{(n)}}{p_{i}^{(n)}}\Big)$ is {\em exponentially convergent} to $({x_1},{x_2})$, if there exists a constant $\alpha>1$ such that 
\begin{align}\label{eq:Dioexp}
\big\|\big({p_{i1}^{(n)}}, {p_{i2}^{(n)}}\big)-{p_{i}^{(n)}}({x_1},{x_2})\big\|_\infty < \big(p_{i}^{(n)}\big)^{1-\alpha}
\end{align}
holds for each $i\in\{0,1,2\}$ and all $n$ large enough. If \eqref{eq:Dioexp} holds for a.e.\ $\bx \in \Delta$, then we say that $(\Delta,f)$ is {\em exponentially convergent}.
\end{definition}

Let $(\Delta,f)$ be a $2$-dimensional continued fraction algorithm. The supremum $\eta^*(\bx)$ of all $\alpha$ satisfying \eqref{eq:Dioexp} for each $i\in\{0,1,2\}$ and all $n$ large enough is called the {\em uniform approximation exponent of $\bx$} for the algorithm $(\Delta,f)$. It is known that, under mild conditions (including ergodicity), the uniform approximation exponent $\eta^*$ of $(\Delta,f)$ satisfies
\begin{equation}\label{eq:bestapproxexp}
\eta^*(\bx) = 1 - \frac{\lambda_2(A)}{\lambda_1(A)}
\end{equation}
for almost all $\bx\in\Delta$. In particular, if $\lambda_2(A)<0$, then $(\Delta,f)$ is exponentially convergent (see for instance \cite[Theorem~4.1]{Lagarias:93}). As stated in \Cref{prop:LagariasLyapunov} below, these conditions are satisfied for the Reverse algorithm. 

We need the following consequence of topological convergence.

\begin{proposition}[{\cite[Theorem~4 in Chapter~3]{Schweiger:00}}]\label{TC}
Let $(\Delta,f)$ be a $2$-dimensional continued fraction algorithm and let $\mathcal{B}^{(s)}$ be the $\sigma$-algebra generated by the cylinders $\Delta(a_0,a_1,\dots, a_{s-1})$ in $\Delta$ of rank~$s$. If $(\Delta,f)$ is topologically convergent everywhere, then
$\bigvee_{s=0}^\infty\mathcal{B}^{(s)}$ equals the $\sigma$-algebra of Borel sets on~$\Delta$. 
\end{proposition}

Recall that a $2$-dimensional continued fraction algorithm $(\Delta,f)$ is called {\em full}, if for each word $a_0\cdots a_{n-1}$ we have $f^n(\Delta(a_0\cdots a_{n-1}))=\Delta$.
To prove ergodicity of the Reverse algorithm we will use the following criterion that goes back to R\'enyi~\cite{Renyi:57}. To be precise, the following proposition is an improvement of R\'{e}nyi's criterion due to Schweiger; details can be found in~\cite[Section~9.5]{Schweiger:95}.
 
\begin{proposition}[R\'enyi 1957; {\cite[Theorems~8 and~6 in Chapter~3]{Schweiger:00}}] \label{Renyi's Theorem}
Let $(\Delta, f)$ be a $2$-dimensional continued fraction algorithm that satisfies the following conditions. 
\begin{itemize}
\item[$(a)$] $\mathrm{Leb}(\Delta)<\infty$. 
\item[$(b)$] The algorithm is full.
\item[$(c)$] $\bigvee_{s=0}^\infty\mathcal{B}^{(s)}$ equals the $\sigma$-algebra of Borel sets on $\Delta$.
\item[$(d)$] There exists a constant $C$ such that for all admissible words $a_0 \cdots a_{n-1}$ we have
\begin{align*}
\sup_{\bx\in \Delta}\omega(a_0 \cdots a_{n-1};\bx)\le C\cdot\inf_{\bx\in \Delta}\omega(a_0 \cdots a_{n-1};\bx),
\end{align*}
where $\omega(a_0 \cdots a_{n-1};\bx)$ is the Jacobian of the inverse branch of $f^n$ that maps onto the cylinder $\Delta(a_0\cdots a_{n-1})$.
\end{itemize}
Then, $f$ is ergodic and admits a finite invariant measure $\mu$ absolutely continuous with respect to Lebesgue measure. This measure is unique up to a scalar factor.
\end{proposition}

Condition (d) of \Cref{Renyi's Theorem} is called {\em R\'enyi's condition} and the constant~$C$ is called {\em R\'enyi constant} for~$f$. 
When applying R\'enyi's criterion, we will use the following representation of the Jacobian of a $2$-dimensional continued fraction algorithm.

\begin{lemma}[Schweiger, {\cite[Lemma~25 in Chapter~11]{Schweiger:00}}]\label{Sch_lemma25}
The Jacobian for a $2$-dimensional continued fraction algorithm $(\Delta,f)$ is given by
\begin{align*}
\omega(a_0\cdots a_{n-1}; \;^t(x_0,x_1,x_2))=\frac{\det{A^{(n)}(\Delta(a_0\cdots a_{n-1}))}}{\big(\|A^{(n)}_0\|_1x_0+\|A^{(n)}_1\|_1x_1+\|A^{(n)}_2\|_1x_2\big)^{3}},
\end{align*}
where $A^{(n)}_0, A^{(n)}_1,A^{(n)}_2$ are the row vectors of $A^{(n)}(\Delta(a_0\cdots a_{n-1}))$.
\end{lemma}

For matrices of determinant $\pm 1$ this lemma is already contained in  \cite[Proposition 5.2]{Veech78}.

\section{Ergodicity}\label{sec:erg}
The Reverse algorithm is a $2$-dimensional continued fraction algorithm in the sense of \Cref{def:cf}. Indeed, it corresponds to the cocycle $A$ given in \eqref{eq:coReverse}. In this section, we prove the ergodicity of the Reverse algorithm $(\Delta, f_\mathrm{R})$ with respect to the measure $\mu$ with density $h$ given in \eqref{eq:density}. 
We first introduce the following notation. Let $A^{(n)}_0, A^{(n)}_1, A^{(n)}_2$ be the row vectors of $A^{n}(\Delta(a_0 \cdots a_{n-1}))$. Choose $\pi_0(n),\pi_1(n),\pi_2(n)$ in a way that $\{\pi_0(n),\pi_1(n),\pi_2(n)\}=\{0,1,2\}$ and  $\|A^{(n)}_{\pi_0(n)}\|_1\le \|A^{(n)}_{\pi_1(n)}\|_1\le \|A^{(n)}_{\pi_2(n)}\|_1$ holds. To simplify notation, we write $A^{(n)}_{\pi_a}$ instead of $A^{(n)}_{\pi_a(n)}$.

When proving ergodicity of $(\Delta, f_\mathrm{R})$, the first idea is to establish the conditions of \Cref{Renyi's Theorem} for $(\Delta, f_\mathrm{R})$ directly. 
To prove condition~(d), let
\begin{align*}
C(a_0\cdots a_{n-1})=\frac{\displaystyle\sup_{\bx\in \Delta}\omega(a_0\cdots a_{n-1};\bx)}{\displaystyle\inf_{\bx\in \Delta}\omega(a_0 \cdots a_{n-1};\bx)}.
\end{align*}
We need to show that $C$ is bounded. Note that the supremum and the infimum in the definition of $C(a_0 \cdots a_{n-1})$ range over the whole set $\Delta$ and recall that $\|A^{(n)}_{\pi_0(n)}\|_1\le \|A^{(n)}_{\pi_1(n)}\|_1\le \|A^{(n)}_{\pi_2(n)}\|_1$. Thus, observing that $\det{A^{(n)}(\Delta(a_0\cdots a_{n-1}))}$ does not depend on $(x_0,x_1,x_2)$, \Cref{Sch_lemma25} yields
\begin{equation}\label{eq:Cquot}
C(a_0  \cdots a_{n-1})=
\frac{\displaystyle\sup_{(x_0,x_1,x_2)\in \Delta}{\big(\|A^{(n)}_0\|_1x_0+\|A^{(n)}_1\|_1x_1+\|A^{(n)}_2\|_1x_2\big)^{3}}}{\displaystyle\inf_{(x_0,x_1,x_2)\in \Delta}{\big(\|A^{(n)}_0\|_1x_0+\|A^{(n)}_1\|_1x_1+\|A^{(n)}_2\|_1x_2\big)^{3}}}
=
\bigg(\frac{\|A^{(n)}_{\pi_2}\|_1}{\|A^{(n)}_{\pi_0}\|_1}\bigg)^3.
\end{equation}
Since
\begin{align*}
^tM^n_{1}=
\begin{pmatrix}
1 & 0 & 0\\
n & 1 & 0\\
n & 0 & 1\\
\end{pmatrix},
\end{align*}
we have $C(1^n)=(n+1)^3$.
Therefore, there exists no R\'{e}nyi constant for all cylinders and, hence, we cannot prove (d) directly. Following the ideas of \cite{Schweiger:00}, we will first show that the conditions of \Cref{Renyi's Theorem} are satisfied for a so-called {\em jump transformation} (see, for instance, \cite[Definition~26 in Chapter~3]{Schweiger:00}). In particular, we consider the  jump transformation $f_{\mathrm{R}}^*$ of $f_{\mathrm{R}}$  on $\Delta(4)$, which avoids $a_n\in\{1,2,3\}$, {\it i.e.}, we set
$
\tau(\bx)=1+\inf\{n\ge0 \,:\, f_{\mathrm{R}}^n(\bx)\in\Delta(4)\}
$
and define
\begin{align*}
f_{\mathrm{R}}^{*}(\bx)=f_{\mathrm{R}}^{\tau(\bx)}(\bx).
\end{align*}
Because the Rauzy gasket has zero Lebesgue measure (see \cite[Section~6.3]{Arnoux_2013}), for Lebesgue almost every $\bx\in\Delta$, $a_n(\bx) =4$ holds for infinitely many choices of $n$. Thus, the jump transformation $f_{\mathrm{R}}^*$ is well-defined a.e.\ on $\Delta$ and satisfies conditions (a) and (b) of \Cref{Renyi's Theorem} by definition. To prove condition (c) for $f_{\mathrm{R}}^*$ on a subset of $\Delta$ of full measure, we require the following result. 

\begin{lemma}\label{lem:topconv}
The dynamical system $(\Delta,f_{\mathrm{R}}^*)$ is topological convergent almost everywhere.
\end{lemma}

\begin{proof}
We first establish a.e.\ topological convergence of $(\Delta, f_{\mathrm{R}})$. By the Poincar\'e recurrence theorem, for almost all $\bx\in\Delta$, the sequence $(M_{a_n})_{n\ge 0}=(\,^t\!A(f_{\mathrm{R}}^{n}(\bx)))_{n\ge 0}$ of matrices is recurrent. Moreover, since the Lebesgue measure of the Rauzy gasket is zero, $M_{a_n}=M_4$ holds for infinitely many $n$ for a.e.\ $\bx$. Thus, by recurrence, there is $\ell\in\N$ and a block $(N_1,\ldots, N_{\ell-1})$ such that $(M_{a_n}, \ldots, M_{a_{n+\ell}})=(M_4,N_1,\ldots,N_{\ell-1},M_4)$ holds for infinitely many $n$. It is easy to check that $M_4N_1\cdots N_{\ell-1}M_4$ is a positive matrix, hence, $(M_{a_n})_{n\ge 0}$ is primitive. Thus, the conditions of \Cref{lem:3.5.5} hold a.e.\ and topological convergence a.e.\ of $(\Delta,f_{\mathrm{R}})$ follows from this proposition. Because $(\Delta,f_{\mathrm{R}}^*)$ is an acceleration of $(\Delta,f_{\mathrm{R}})$, it is {\em a fortiori} topologically convergent a.e.\ as well, and the lemma is proved.
\end{proof}

An alternative proof of Lemma~\ref{lem:topconv} 
could be done by using ideas from \cite{Messaoudi2009}. Indeed, for each word $w_1\cdots w_k$ ending with the letter $4$ and containing two other digits, one can show that the cylinder $\Delta(a_0 \cdots a_{n-1}w_1\cdots w_k)$ has at most $\frac{2}{3}$ times the diameter of $\Delta(a_0\cdots a_{n-1})$. Since our proof is much shorter, we refrain from giving the technical details.

To get condition (d) of \Cref{Renyi's Theorem} for $f_{\mathrm{R}}^*$, we start with an auxiliary lemma.

\begin{lemma}
\label{A0+A1>A2}
We have $\|A^{(n)}_{\pi_0}\|_1+\|A^{(n)}_{\pi_1}\|_1>\|A^{(n)}_{\pi_2}\|_1$.
\end{lemma}
\begin{proof}
We prove this by induction. Since the induction start trivially holds, we continue with the induction step. By the symmetry properties of this algorithm, we may restrict ourselves to $a_{n}\in\{1,4\}$. For $a_{n}=1$, we start with 
\[
^t\!A^{n}(\Delta(a_0 \cdots a_{n-1},1))=M_{a_0}\cdots M_{a_{n-1}}M_{1}
=\big(A^{(n)}_0, A^{(n)}_0+A^{(n)}_1, A^{(n)}_0+A^{(n)}_2
\big).
\]
Since all occurring matrices are positive, the $1$-norm behaves nicely w.r.t.\ addition, and we gain
\begin{align*}
\|A^{(n+1)}_{\pi_0}\|_1+\|A^{(n+1)}_{\pi_1}\|_1
&=\|A^{(n)}_{0}\|_1+(\|A^{(n)}_{0}\|_1+\min\{\|A^{(n)}_{1}\|_1, ~\|A^{(n)}_{2}\|_1\})\\
&\ge 2\|A^{(n)}_{\pi_0}\|_1+\|A^{(n)}_{\pi_1}\|_1
>\|A^{(n)}_{\pi_0}\|_1+\|A^{(n)}_{\pi_2}\|_1
=\|A^{(n+1)}_{\pi_2}\|_1.
\end{align*}
Similarly, for $a_{n}=4$, we get from
\begin{align*}
M_{a_1}\cdots M_{a_{n-1}}M_{4}
=\big(A^{(n)}_1+A^{(n)}_2, A^{(n)}_0+A^{(n)}_2, A^{(n)}_0+A^{(n)}_1\big)
\end{align*}
that
\begin{align*}
\|A^{(n+1)}_{\pi_0}\|_1+\|A^{(n+1)}_{\pi_1}\|_1&=\|A^{(n)}_{\pi_0}\|_1+\|A^{(n)}_{\pi_1}\|_1+\|A^{(n)}_{\pi_0}\|_1+\|A^{(n)}_{\pi_2}\|_1\\
&>\|A^{(n)}_{\pi_1}\|_1+\|A^{(n)}_{\pi_2}\|_1=\|A^{(n+1)}_{\pi_2}\|_1. \qedhere
\end{align*}
\end{proof}

The next lemma provides a R\'{e}nyi constant for cylinders $\Delta(a_0\cdots a_{n})$ with $a_n=4$. 

\begin{lemma}\label{an=4}
Renyi's condition is satisfied for $\omega(a_0 \cdots a_{n-1}a_n;x)$ if $a_{n}=4$.
In particular, we have
$C(a_0 \cdots a_{n-1}a_n)<2^3$.
\end{lemma}
\begin{proof}
Let $M_{a_0}\cdots M_{a_{n-1}}=
\big(A_0^{(n)}, A_1^{(n)}, A_2^{(n)}\big)$. Then, 
\begin{align*}
M_{a_1}\cdots M_{a_{n-1}}M_{4}
=\big(A_1^{(n)}+A_2^{(n)}, A_0^{(n)}+A_2^{(n)}, A_0^{(n)}+A_1^{(n)}\big).
\end{align*}
Thus, by  \eqref{eq:Cquot} and because $\|A^{(n)}_{\pi_0}\|_1\le \|A^{(n)}_{\pi_1}\|_1\le \|A^{(n)}_{\pi_2}\|_1$, we gain 
\begin{equation}\label{eq:newCC}
C(a_0 \cdots a_{n-1}a_n)
=\left(\frac{\|A_{\pi_1}^{(n)}+A_{\pi_2}^{(n)}\|_1}{\|A_{\pi_0}^{(n)}+A_{\pi_1}^{(n)}\|_1}\right)^3
=\left(\frac{\|A_{\pi_1}^{(n)}\|_1+\|A_{\pi_2}^{(n)}\|_1}{\|A_{\pi_0}^{(n)}\|_1+\|A_{\pi_1}^{(n)}\|_1}\right)^3.
\end{equation}
Since we have $\|A^{(n)}_{\pi_0}\|_1+\|A^{(n)}_{\pi_1}\|_1>\|A^{(n)}_{\pi_2}\|_1$ by \Cref{A0+A1>A2}, we conclude from \eqref{eq:newCC} that
\[
C(a_0 \cdots a_{n-1}a_n)<\left(\frac{\|A_{\pi_2}^{(n)}\|_1+\|A_{\pi_2}^{(n)}\|_1}{\|A_{\pi_2}^{(n)}\|_1}\right)^3
=2^3. \qquad \qedhere
\]
\end{proof}

\begin{proposition}\label{4infinitely}
The jump transformation $f_{\mathrm{R}}^*$ satisfies the R\'enyi condition.
\end{proposition}

\begin{proof}
This is an immediate consequence of \Cref{an=4}.
\end{proof}

We are now ready to prove our ergodicity result.

\begin{proof}[Proof of $\Cref{ergodicity}$]
We first prove ergodicity of $(\Delta, f_{\mathrm{R}}^*)$.  To this end, we check the conditions of \Cref{Renyi's Theorem}. By definition, $(\Delta, f_{\mathrm{R}}^*)$ satisfies conditions (a) and (b). It follows from \Cref{lem:topconv} that the dynamical system $(\Delta, f_{\mathrm{R}}^*)$ is topological convergent a.e. Thus, the cylinders generate the $\sigma$-algebra of Borel sets on a subset $\tilde\Delta$ of $\Delta$ of full measure, and, hence, by \Cref{TC}, condition~(c) holds on $\tilde \Delta$. Finally, by \Cref{4infinitely}, the system $(\Delta, f_{\mathrm{R}}^*)$ satisfies Renyi's condition and, hence, condition (d). Thus, by \Cref{Renyi's Theorem}, $f_{\mathrm{R}}^*$ is ergodic (on $\tilde\Delta$ and, hence, on $\Delta$) and admits a finite invariant measure that is equivalent to the Lebesgue measure. Since ergodicity of $f_{\mathrm{R}}^*$ implies ergodicity of $f_{\mathrm{R}}$ with respect to $\mu$ by {\cite[Theorem~18.2.3]{Schweiger:95}}, observe the unicity of the measure up to scalar multiples stated in \cite[Theorem~6 in Chapter~3]{Schweiger:00}, the proof is finished.
\end{proof}

\section{The second Lyapunov exponent}\label{sec:lyap}
This section is devoted to the proof of Theorem~\ref{second Lyapunov negative}. Let $A$ be the cocycle of an ergodic $2$-dimensional continued fraction algorithm $(\Delta, f)$. As mentioned in Section~\ref{sec:MUl}, the Lyapunov exponents $\lambda_1(A)\ge \lambda_2(A)\ge \lambda_3(A)$ are related to the convergence behavior of the  algorithm $(\Delta, f)$. In particular, under mild conditions, $\lambda_2(A)<0$ implies exponential convergence of the algorithm. Since it is easier to estimate the largest Lyapunov exponent of a cocycle rather than estimating the second largest one, in a first step we will define a cocycle $D$ whose largest Lyapunov exponent $\lambda_1(D)$ contains information on the quality of the approximations of the algorithm. 
The idea for this cocycle goes back to \cite{Lagarias:93}, however, since our algorithm is unsorted and its projectivized version is normalized by the $1$-norm, the details are different in our situation. 

Let $(\Delta, f)$ be a $2$-dimensional continued fraction algorithm with cocycle $A$. For $n\in \N$, let the entries of the cocycle matrix $A^{(n)}(\bx)$ be given as in \eqref{eq:Acomponents} and recall that $p_{i}^{(n)}=p_{i0}^{(n)}+p_{i1}^{(n)}+p_{i2}^{(n)}$ for $i\in\{0,1,2\}$.  For $\bx=\,^t(x_0,x_1,x_2)\in\Delta$ define the matrices
\begin{align*}
\Pi&=\begin{pmatrix}
-1 & 1 & 0\\
-1 & 0 & 1\\
\end{pmatrix}, \quad 
H(\bx)=
\begin{pmatrix}
-{x_1} & -{x_2}\\
1-{x_1} & -{x_2}\\
-{x_1} & 1-{x_2}
\end{pmatrix},
\end{align*}
and
\begin{align*}
D^{(n)}(\bx)=\Pi A^{(n)}(\bx)H(\bx)
=
\begin{pmatrix}
p_{11}^{(n)}-p_{01}^{(n)}-(p_{1}^{(n)}-p_{0}^{(n)})x_1 &
p_{12}^{(n)}-p_{02}^{(n)}-(p_{1}^{(n)}-p_{0}^{(n)})x_2\\
p_{21}^{(n)}-p_{01}^{(n)}-(p_{2}^{(n)}-p_{0}^{(n)})x_1 & 
p_{22}^{(n)}-p_{02}^{(n)}-(p_{2}^{(n)}-p_{0}^{(n)})x_2
\end{pmatrix}.
\end{align*}
It turns out that $D^{(n)}(\bx)$ has the cocycle property. To show this, we first claim that 
\begin{equation}\label{eq:HPweg}
H(f^n \bx)\, \Pi\, A^{(n)}(\bx) H(\bx) = A^{(n)}(\bx) H(\bx).
\end{equation}
To prove this, let $I_{3\times 3}$ be the $3\times 3$ identity matrix, let $\mathbf{0}_{3\times 3}$ be the $3\times 3$ zero matrix, and let $\mathbf{0}_{3}$ be the zero vector in $\mathbb{R}^3$. Observe that $I_{3\times3} - H(f^n \mathbf{x})\, \Pi=\,^t(f^n(\bx),f^n(\bx), f^n(\bx))$ and $^t\!f^n(\bx) A^{(n)}(\bx) H(\bx) = \,^t\bx\, H(\mathbf{x}) = \mathbf{0}_3$, and, hence,
\[
(I_{3\times 3} - H(f^n \mathbf{x})\, \Pi)A^{(n)}(\bx) H(\bx)=\,^t(f^n(\bx),f^n(\bx), f^n(\bx))A^{(n)}(\bx) H(\bx)=
\,^t(\bx,\bx,\bx)H(\bx) = \mathbf{0}_{3\times 3}, 
\]
which implies~\eqref{eq:HPweg}. Using~\eqref{eq:HPweg} we obtain that
\[
D^{(m)}(f^n\bx) D^{(n)}(\bx) = \Pi\, A^{(m)}(f^n\bx) A^{(n)}(\bx) H(\bx)=\Pi \, A^{(m+n)}(\bx) H(\bx) = D^{(m+n)}(\bx),
\]
thus $D$ is a cocycle of $(\Delta,f)$. 

In the following result, for two sequences $(a_n)_{n\in\N}$ and $(b_n)_{n\in\N}$ the notation $a_n \ll b_n$ indicates that there is $C>0$ and $N\in\N$ such that $a_n \le Cb_n$ holds for every $n\ge N$. Recall that $\delta_1(M)\ge \delta_2(M)\ge \delta_3(M)$ are the singular values of a $3\times 3$ matrix $M$.

\begin{lemma}\label{lem.singular value}
Let $(\Delta, f)$ be an ergodic $2$-dimensional continued fraction algorithm with cocycle~$A$. 
     Then $\delta_2(A^{(n)}(\bx))\ll \delta_1(D^{(n)}(\bx))$ holds uniformly for all $\bx\in\Delta$, and $\lambda_2(A)\le \lambda_1(D)$.
\end{lemma}

The reverse inequalities can also be shown, but are not needed for our purposes. 

\begin{proof}
Let $\|\cdot\|$ be an arbitrary norm. Let $\bx\in\Delta$ and recall that $D^{(n)}(\bx)=\Pi A^{(n)}(\bx)H(\bx)$. As in the proof of \cite[Lemma~4.2]{BST21}, we can establish the lemma by mapping the unit circle $\mathbb{S}^{1}$ into $\mathbb{R}^2$ step by step via the matrices $H(\bx)$, $A^n(\bx)$, and $\Pi$, and keep track the lengths of the semi-axes of the occurring ellipses. First, setting $\bx=\,^t(x_0,x_1,x_2)$, we see that for $^t(v_1,v_2)\in \mathbb{S}^1$ we get
\begin{align*}
    H(\bx) \cdot \begin{pmatrix}
        v_1 \\ 
        v_2
    \end{pmatrix}
    =
    \begin{pmatrix}
        -x_1v_1-x_2v_2 \\ 
        v_1-x_1v_1-x_2v_2 \\ 
        v_2-x_1v_1-x_2v_2 \\ 
    \end{pmatrix}.
\end{align*}
Thus $E_1= H(\bx)\mathbb{S}^1$ is an ellipse whose semi-axes  $\ba_1^{(1)},\ba_2^{(1)}$ satisfy $1 \ll \|\ba_1^{(1)}\|$ and $1 \ll\|\ba_2^{(1)}\|$, where the implied constants can be chosen independently of $\bx$. Because $\bx\in\Delta$ we see that $E_1$ is contained in the orthogonal complement $\bx^\bot$ of $\bx$. By definition, $A^{(n)}(\bx)$ maps $E_1$ to an ellipse $E_2=A^{(n)}(\bx)H(\bx)\mathbb{S}^1 \subset \by^\bot$, where $\by = f^n(\bx) \in \Delta$. Moreover, by the definition of the singular value of a matrix, this entails that the semi-axes $\ba_1^{(2)},\ba_2^{(2)}$ of $E_2$ satisfy
\begin{equation}\label{eq:E2E3}
\delta_{2}(A^{(n)}(\bx)) \ll \|\ba_1^{(2)}\|, \quad\text{and}\quad \delta_{3}(A^{(n)}(\bx)) \ll \|\ba_2^{(2)}\|,
\end{equation}
where the implied constants can again be chosen independently of $\bx$.
Thus it remains to show that $\|\Pi \bw\|_\infty\ge \frac13$ for each $\bw=\,^t(w_0,w_1,w_2) \in \by^\bot$ with $\|\bw\|_\infty=1$. To see this, note that 
\begin{align*}
    \Pi \cdot \begin{pmatrix}
        w_0 \\ 
        w_1 \\
        w_2
    \end{pmatrix}
    =
    \begin{pmatrix}
        w_1-w_0 \\ 
        w_2-w_0
    \end{pmatrix}.
\end{align*}
Assume on the contrary that $\|\Pi \bw\|_\infty<\frac13$ for some $\bw=\,^t(w_0,w_1,w_2) \in \by^\bot$ with $\|\bw\|_\infty=1$.
Because $\|\bw\|_\infty=1$ we have $|w_i| = 1$ for at least one $i\in\{0,1,2\}$. Moreover, $\|\Pi \bw\|_\infty<\frac13$ implies that $|w_j-w_k|<\frac23$ for all $j,k\in\{0,1,2\}$. Therefore, because $\by=\,^t(y_0,y_1,y_2)\in \Delta$,
\[ 
|y_0w_0+y_1w_1+y_2w_2| \ge |y_0+y_1+y_2|-\frac23|y_0+y_1+y_2|=1-\frac23=\frac13.
\]
This contradicts the fact that $\bw\in \by^\bot$ and, hence, $\|\Pi \bw\|_\infty\ge \frac13$ for each $\bw \in \by^\bot$ with $\|\bw\|_\infty=1$. Together with \eqref{eq:E2E3} this implies that the ellipse $E_3=\Pi A^{(n)}(\bx)H(\bx)\mathbb{S}^1 = D^{(n)}(\bx)\mathbb{S}^1$ has semi-axes  $\ba_1^{(3)},\ba_2^{(3)}$ satisfying
$\delta_{2}(A^{(n)}(\bx)) \ll \|\ba_1^{(3)}\|$ and $\delta_{3}(A^{(n)}(\bx)) \ll \|\ba_2^{(3)}\|$,
where the implied constants can again be chosen independently of $\bx$. This proves the statement on the singular values. By ergodicity and by the definition of Lyapunov exponents, this implies that $\lambda_2(A)\le \lambda_1(D)$.
\end{proof}

Let $(\Delta, f)$ be an ergodic $2$-dimensional continued fraction algorithm. We relate exponential convergence of $(\Delta, f)$ to the cocycle $D$ as follows.  Let $\bx=\,^t(x_0,x_1,x_2)\in\Delta$ be arbitrary. By the mapping properties of the matrix $\Pi$ on $\by^\bot$ for $\by\in\Delta$ that we discussed in the proof of Lemma~\ref{lem.singular value}, we know that 
\begin{align*}
&\max\big\{\big\|\big({p_{i1}^{(n)}}, {p_{i2}^{(n)}}\big)-{p_{i}^{(n)}}({x_1},{x_2})\big\|_\infty\;:\; i\in\{0,1,2\}\big\}
\qquad\hbox{and} \\
&\max\big\{\big\|\big({p_{i1}^{(n)}-p_{01}^{(n)}}, {p_{i2}^{(n)}-p_{02}^{(n)}}\big)-(p_{i}^{(n)}-p_0^{(n)})({x_1},{x_2})\big\|_\infty\;:\; i\in\{1,2\}\big\}
\end{align*}
are at most a factor $\frac13$ apart from each other. Thus, the criterion for exponential convergence contained in \eqref{eq:Dioexp} is equivalent to the fact that 
\begin{align}\label{eq:Dioexp2}
\big\|\big({p_{i1}^{(n)}-p_{01}^{(n)}}, {p_{i2}^{(n)}-p_{02}^{(n)}}\big)-(p_{i}^{(n)}-p_0^{(n)})({x_1},{x_2})\big\|_\infty < \big(p_{i}^{(n)}\big)^{1-\alpha}
\end{align}
holds for each $i\in\{1,2\}$ when $n$ is large enough. Therefore, the uniform approximation exponent $\eta^*$ defined in \eqref{eq:bestapproxexp} can also be defined in terms of the entries of $D^{(n)}(\bx)$. 

According to the following proposition, the uniform approximation exponent of the Reverse algorithm can be expressed in terms of its Lyapunov exponents (see for instance \cite[Theorem~4.1]{Lagarias:93}, whose conditions are satisfied for the Reverse algorithm according to the results that we established in Section~\ref{sec:erg}).

\begin{proposition}\label{prop:LagariasLyapunov}
Let $\lambda_1(A)\ge \lambda_2(A)\ge\lambda_3(A)$ be the Lyapunov exponents of the cocycle $A$ of the Reverse algorithm $(\Delta,f_\mathrm{R})$. Then the uniform approximation exponent $\eta^*$ of $(\Delta,f_\mathrm{R})$ satisfies
\[
\eta^*(\bx) = 1 - \frac{\lambda_2(A)}{\lambda_1(A)}
\]
for almost all $\bx\in\Delta$. In particular, if $\lambda_2(A)<0$, then the Reverse algorithm is a.e.\ exponentially convergent.
\end{proposition}

In order to prove Theorem~\ref{second Lyapunov negative}, we have to establish a negative upper bound for $\lambda_2(A)$ for the Reverse algorithm. By Lemma~\ref{lem.singular value}, it suffices to provide such a bound for the first Lyapunov exponent $\lambda_1(D)$ of the cocycle $D$. Since the 1-norm is submultiplicative, the cocycle property of $D$ implies that
\begin{align}\label{subadditivity relation}
\log{\|D^{(n+m)}(\bx)\|_1}\le\log{\|D^{(m)}(f^{n}\bx)\|_1}+\log{\|D^{(n)}(\bx)\|_1}.
\end{align}
This subadditivity property of $D$ allows us to apply the following classical result.

\begin{proposition}[Kingman's Subadditive Ergodic Theorem]\label{prop:King}
Let $T$ be a measure-preserving transformation on the probability space $(\Omega ,\nu )$, and let $(g_{n})_{{n\in {\mathbb  {N}}}}$ be a sequence of 
$L^{1}$-functions such that $g_{{n+m}}(x)\leq g_{n}(x)+g_{m}(T^{n}x)$ (subadditivity relation). Then
$\lim_{n\rightarrow \infty}\frac{g_n(x)}{n}=:g(x)$ holds $\nu$-a.e.,
where $g(x)$ is $T$-invariant. If $T$ is ergodic, then $g(x)$ is constant.
\end{proposition}

Using \Cref{prop:King},~\eqref{subadditivity relation}, and the equivalent definitions of the first Lyapunov exponent in~\eqref{eq:lyapudef}, we gain 
\begin{align*}
\lambda_1(D)=\inf_{n\in\mathbb{N}}\frac1n\int_{\Delta}\log{\|D^{(n)}(\bx)\|_1}d\mu(\bx),
\end{align*}
and, hence, the following result holds.

\begin{lemma}[{\cite[Theorem~3]{Hardcastle:02}}]\label{lem:HaEst}
Let $(\Delta, f)$ be an ergodic $2$-dimensional continued fraction algorithm with cocycle $A$. We have $\lambda_2(A)< 0$ if and only if there exists $n\in\mathbb{N}$
such that
\begin{align}\label{eq:l2int}
\frac{1}{n}\int_{\Delta}\log{\|D^{(n)}(\bx)\|_1d\mu(\bx)}<0.
\end{align}
\end{lemma}

We are now ready to give the proof of ~\Cref{second Lyapunov negative}.

\begin{proof}[Proof of~\Cref{second Lyapunov negative}]
We split the integral in \eqref{eq:l2int} as
\begin{align*}
\frac{1}{n}\int_{\Delta}\log{\|D^{(n)}(\bx)\|_1d\mu(\bx)}&=I_1(n)+I_2(n)
\end{align*}
with
\begin{align*}
I_1(n)=\frac{1}{n}\int_{\bigcup_{i=1}^{3}\Delta(i^{n})}\log{\|D^{(n)}(\bx)\|_1d\mu(\bx)}, \quad
I_2(n)=\frac{1}{n}\int_{\Delta \setminus \bigcup_{i=1}^{3}\Delta(i^{n})}\log{\|D^{(n)}(\bx)\|_1d\mu(\bx)}.
\end{align*}
According to \Cref{lem:HaEst}, we have to estimate $I_1(n)$ and $I_2(n)$ for a suitable $n\in \N$.
For $n\in\N$, we set 
\begin{align*}
\mathbb{L}_1(n):&=\frac{1}{n}\sum_{i=1}^3\max_{\bx\in\Delta(i^{n})}\log{\|D^{(n)}(\bx)\|_1}\int_{\Delta(i^{n})}d\mu
\\&=\frac{3}{n}\log{\Big(\frac{2n(n-1)}{n+1}-1\Big)}\int_{0}^{\frac{1}{n+1}}\int_{0}^{\frac{1}{n+1}}\frac{4}{\pi^2(1-x)(1-y)(x+y)}dydx
\end{align*}
and
\begin{align*}
\mathbb{L}_2(n):=&\frac{4}{\pi^2n}\sum_{\substack{w\in\{1,2,3,4\}^n\setminus\bigcup_{i=1}^{3}\{i\}^n,\\ \max\{\log{\|D^{(n)}(\bx)\|_1}\,:\,\bx\in\Delta(w) \}>0}}
\bigg(\prod_{j=0}^2\max\Big\{\frac{1}{(1-x_j)}\,:\, \bx\in \Delta(w)\Big\}\bigg)
\\
&\hskip5cm\cdot \mathrm{Leb}(\Delta(w))\max_{\bx\in\Delta(w)}\log{\|D^{(n)}(\bx)\|_{1}}\\
&+\frac{4}{\pi^2n}\sum_{\substack{w\in\{1,2,3,4\}^n\setminus\bigcup_{i=1}^{3}\{i\}^n, \\ \max\{\log{\|D^{(n)}(\bx)\|_1}\,:\,\bx\in\Delta(w) \}<0}}\bigg(\prod_{j=0}^2\min\Big\{\frac{1}{(1-x_j)}\,:\, \bx\in \Delta(w)\Big\}\bigg)
\\
&\hskip5.2cm\cdot \mathrm{Leb}(\Delta(w))\max_{\bx\in\Delta(w)}\log{\|D^{(n)}(\bx)\|_1}.
\end{align*}
The sum of the integrals of the maximum of $\log{\|D^{(n)}(\bx)\|_{1}}$ on each cylinder set $\Delta(w)$ with $w\in\{1,2,3,4\}^n\setminus\bigcup_{i=1}^{3}\{i\}^n$, is larger than or equal to $I_2$. Furthermore, when the maximum of $\log{\|D^{(n)}(\bx)\|_{1}}$ is positive, choosing the largest density of the measure $\mu$ on each cylinder yields a quantity larger than or equal to $I_2$; when the maximum of one is negative, selecting the smallest density of the measure $\mu$ on each cylinder yields a quantity larger than or equal to $I_2$.
Taking this into account we conclude that $I_j(n) \le \mathbb{L}_j(n)$
for $j\in\{1,2\}$ and $n\in\N$.
We got the following estimates by computer calculation (with floating point error handling):
\begin{align}\label{eq:Lest}
\mathbb{L}_1(12)<0.024002, \quad \mathbb{L}_2(12)<-0.044610.
\end{align}
Therefore,
\begin{align*}
\frac{1}{12}\int_{\Delta}\log{\|D^{(12)}(\bx)\|_1}d\mu(\bx)<\mathbb{L}_1(12)+\mathbb{L}_2(12) < -0.020608,
\end{align*}
and the proof is finished.
\end{proof}

\begin{remark}[Precision of floating point calculations]
We calculated each term in $\mathbb{L}_1(12)$ and $\mathbb{L}_2(12)$ with $18$-digit precision and then added all of them. 
Each term contributes an error bounded by $<10^{-18}$, and the number of summands equals $4^{12}<10^8$ terms, so the total error is bounded by $10^{-10}$. This error is already added in the estimates of $\mathbb{L}_1(12)$ and $\mathbb{L}_2(12)$ provided in \eqref{eq:Lest}.
\end{remark}

\section{Balancedness}\label{sec:bal}
In this section, we will prove the balancedness result contained in \Cref{th:balance}.
Roughly speaking, we show that balance properties of a set of substitutions (in the sense of \Cref{lem:sigma_ibalance}) lead to balance properties of the languages of the directive sequences formed by these substitutions. 
We start with two results on the behavior of letter balancedness under the action of the substitutions~\eqref{eq:reversSubst} of the Reverse algorithm. Recall that a pair $(v_1,v_2)\in\cA^*\times \cA^*$ of words is called {\em $C$-letter balanced}, if for each $i\in\cA$ we have $\big| |v_1|_i-|v_2|_i\big| \le C$.

\begin{lemma}\label{lem:presigma_ibalance}
Let $\sigma\in\{\sigma_1,\sigma_2,\sigma_3\}$ and let $(v_1,v_2)\in \cA^*\times\cA^*$ with $|v_1|=|v_2|$ be $C$-letter balanced. If $|\sigma(v_1)| \le |\sigma(v_2)|$ and if $w$ is a factor of $\sigma(v_2)$ with $|w|=|\sigma(v_1)|$, then the pair $(\sigma(v_1),w)$ is $C$-letter balanced.
\end{lemma}

\begin{proof}
By symmetry, we may assume w.l.o.g.\ that $\sigma=\sigma_1$. 
Because $|\sigma(v_1)| \le |\sigma(v_2)|$, $C$-letter balancedness of $(v_1,v_2)$ implies that there is $d_1\in\{0,\ldots,C\}$ such that $|v_1|_1=|v_2|_1+d_1$ and, hence, $|\sigma(v_1)| +d_1= |\sigma(v_2)|$. By $C$-letter balancedness of $(v_1,v_2)$ and because $|v_1|=|v_2|$, there is $d_2\in \{-C,\ldots,C-d_1\}$ such that
\[
\begin{split}
|v_1|_2&=|v_2|_2+d_2, \\
|v_1|_3&=|v_2|_3-d_1-d_2.
\end{split}
\]
By the definition of $\sigma=\sigma_1$, this yields
\[
\begin{split}
|\sigma(v_1)|_1 &= |\sigma(v_2)|_1, \\
|\sigma(v_1)|_2 &= |\sigma(v_2)|_2 +d_2, \\
|\sigma(v_1)|_3 &= |\sigma(v_2)|_3 -d_1-d_2.
\end{split}
\]
Because $w\prec \sigma(v_2)$ with $|w|=|\sigma(v_2)|-d_1$ this implies 
\[
\begin{split}
|w|_1 \le |\sigma(v_1)|_1 &\le |w|_1+d_1, \\
d_2+|w|_2 \le |\sigma(v_1)|_2 &\le |w|_2+d_1+d_2, \\
-d_1-d_2+|w|_3 \le |\sigma(v_1)|_3 &\le |w|_3-d_2.
\end{split}
\]
Since $|d_1|\le C$, $|d_2|\le C$, and $|d_1+d_2| \le C$ the result follows.
\end{proof}

\begin{proposition}\label{lem:sigma_ibalance}
Let $u\in \cA^*$ be given. If $u$ is $C$-letter balanced, then $\sigma_i(u)$ is $(C+4)$-letter balanced for $i\in\{1,2,3\}$ and $(C+2)$-letter balanced for $i=4$.
\end{proposition}

\begin{proof}
First we consider the case $i\in\{1,2,3\}$ and set $\sigma=\sigma_i$. We have to show that each pair of factors $(w_1,w_2)$ of $\sigma(u)$ with $|w_1|=|w_2|$ is $(C+4)$-letter balanced. There are factors $v_1,v_2$ of $u$ and words $a_1,a_2,b_1,b_2$ of length at most $1$ such that
\begin{equation}\label{eq:w1w2bala1}
w_1=a_1\sigma(v_1)b_1 \quad\text{and}\quad w_2=a_2\sigma(v_2)b_2.
\end{equation}
Thus $-2 \le |\sigma(v_1)|-|\sigma(v_2)| \le 2$ and we may assume w.l.o.g.\ that $|\sigma(v_1)| \le |\sigma(v_2)|$. We claim that $(\sigma(v_1),\sigma(v_2))$ is $(C+2)$-letter balanced. To prove this claim we distinguish two cases.

Case~(i):  $|v_1| < |v_2|$. In this case there exists $v_1'$ with $v_1 \prec v_1' \prec u$ and $|v_1'| =|v_2|$. 
If $|\sigma(v_1')| < |\sigma(v_2)|$ the only possibility is $|\sigma(v_1')| = |\sigma(v_1)|+1$ and $|\sigma(v_2)| = |\sigma(v_1)|+2$. We may apply \Cref{lem:presigma_ibalance} to see that for a factor $w$ of $\sigma(v_2)$ with $|w|=|\sigma(v_1')|$ the pair $(\sigma(v_1'),w)$ is $C$-letter balanced. Because $|\sigma(v_1)|+1= |\sigma(v_1')| = |w| < |\sigma(v_2)| = |\sigma(v_1)|+2$, this implies that $(\sigma(v_1),\sigma(v_2))$ is $(C+2)$-letter balanced.
If $|\sigma(v_1')| \ge |\sigma(v_2)|$, we may again apply \Cref{lem:presigma_ibalance} to see that for a factor $w$ of $\sigma(v_1')$ with $|w|=|\sigma(v_2)|$ the pair $(w,\sigma(v_2))$ is $C$-letter balanced. Because $|\sigma(v_1)|\le |w| \le |\sigma(v_1)|+2$ we may choose $w$ in a way that $\sigma(v_1) \prec w$ and, hence, $(\sigma(v_1),\sigma(v_2))$ is $(C+2)$-letter balanced.

Case~(ii):  $|v_1| \ge |v_2|$. In this case there exists $v_2'$ with $v_2 \prec v_2' \prec u$ and $|v_2'| =|v_1|$. Since $|\sigma(v_1)| \le |\sigma(v_2)| \le |\sigma(v_2')|$, we may apply \Cref{lem:presigma_ibalance} to see that for a factor $w$ of $\sigma(v_2')$ with $|w|=|\sigma(v_1)|$ the pair $(\sigma(v_1),w)$ is $C$-letter balanced. If we choose $w$ in a way that $w\prec \sigma(v_2)$ it follows from $|\sigma(v_2)|-2 \le |w|\le |\sigma(v_2)|$ that  $(\sigma(v_1),\sigma(v_2))$ is $(C+2)$-letter balanced.

Summing up we proved that $(\sigma(v_1),\sigma(v_2))$ is $(C+2)$-letter balanced. Thus we see from~\eqref{eq:w1w2bala1} that the pair $(w_1,w_2)$ is $(C+4)$-letter balanced and the result is proved for $i\in\{1,2,3\}$.

It remains to deal with the case $i=4$. We have to show that each pair of factors $(w_1,w_2)$ of $\sigma_4(u)$ with $|w_1|=|w_2|$ is $(C+2)$-letter balanced. There are factors $v_1,v_2$ of $u$ with $|v_1|=|v_2|$ and words $a_1,a_2,b_1,b_2$ with $|a_1| \le 1$, $|b_1|\le 1$, $|a_2|+|b_2|\le 2$ such that
\begin{equation}\label{eq:w1w2bala2}
w_1=a_1\sigma_4(v_1)b_1 \quad\text{and}\quad w_2=a_2\sigma_4(v_2)b_2.
\end{equation}
Thus it is enough to show that $(\sigma_4(v_1),\sigma_4(v_2))$ is $C$-letter balanced.
By the definition of $\sigma_4$ we have
\[
|\sigma_4(v_\ell)|_1 = |v_\ell|_2 + |v_\ell|_3, \quad
|\sigma_4(v_\ell)|_2 = |v_\ell|_3 + |v_\ell|_1, \quad
|\sigma_4(v_\ell)|_3 = |v_\ell|_1 + |v_\ell|_2 \qquad (\ell\in\{1,2\}),
\]
which implies that 
$|\sigma_4(v_1)|_i - |\sigma_4(v_2)|_i=|v_1|_j - |v_2|_j+|v_1|_k-|v_2|_k=|v_2|_i-|v_1|_i\in[-C,C]$ for all $i\in\{1,2,3\}$ and $\{i,j,k\}=\{1,2,3\}$.
Thus we see from~\eqref{eq:w1w2bala2} that the pair $(w_1,w_2)$ is $(C+2)$-letter balanced because $(\sigma_4(v_1),\sigma_4(v_2))$ is $C$-letter balanced.
\end{proof}

Balancedness of a word $w\in\cA^*$ is known to be related to projection properties of its abelianization $\mathbf{l}(w)$. To state the according results we use notation from \cite{BST:19}. In particular, for $\bv,\bw\in \mathbb{R}^3 \setminus\{\mathbf{0}_3\}$ we denote the projection to $\bw^\bot$ along $\bv$ by~$\pi_{\bv,\bw}$, i.e.,
\begin{equation}\label{eq:projections}
\begin{aligned}
\pi_{\bv,\bw}:\ & \mathbb{R}^3 \to \bw^\perp, & \bx & \mapsto  \bx - \frac{\langle \bx, \bw \rangle}{\langle\bv ,\bw\rangle}\bv. 
\end{aligned}
\end{equation}
Set $\bone=\,^t(1,1,1)$.
We continue with the relation between letter balancedness and projections of abelianizations.
It is an immediate consequence of the proof of \cite[Lemma~4.1]{BST:19} or \cite[Lemma~2]{Delecroix-Hejda-Steiner}. 

\begin{lemma}[{see \cite[Lemma~4.1]{BST:19} or \cite[Lemma~2]{Delecroix-Hejda-Steiner}}]\label{lem:sadiclem41}
Let $C>0$ be given.
\begin{itemize}
\item[(i)]
Let $\btau$ be a primitive sequence of substitutions with generalized right eigenvector $\bu$.
\begin{itemize}
\item  If $\cL_{\btau}$ is $C$-letter balanced, then
$\Vert \pi_{\bu,\bone}(\mathbf{l}(u)) \Vert_\infty \le C$ for each $u\in \cL_{\btau}$. 
\item If $\Vert \pi_{\bu,\bone}(\mathbf{l}(u)) \Vert_\infty \le C$ for each $u\in \cL_{\btau}$, then $\cL_{\btau}$ is $2C$-letter balanced.
\end{itemize}
\item[(ii)]
Let $u\in \cA^\N$ be given.
\begin{itemize}
\item  If $u$ is $C$-letter balanced, then there is $\bu\in\mathbb{R}^3_{\ge 0}$  with $\Vert \pi_{\bu,\bone}(\mathbf{l}(v)) \Vert_\infty \le C$ for each $v\prec u$. 
\item If $\Vert \pi_{\bu,\bone}(\mathbf{l}(p)) \Vert_\infty \le C$ for some $\bu\in\mathbb{R}^3_{\ge 0}$ and each prefix $p$ of $u$, then $u$ is $4C$-letter balanced.
\end{itemize}
\end{itemize}
If (ii) holds, we say that $\bu$ is the {\em frequency vector} of the letter balanced sequence $u$.
\end{lemma} 

Let $\btau=(\tau_n)_{n\ge 0} \in\{\sigma_1,\sigma_2,\sigma_3,\sigma_4\}^\N$ be a directive sequence  with generalized right eigenvector $\bu$. In all that follows, we will denote the associated sequence of incidence matrices by $(\bar B_n)_{n\in\N}=( B_{\tau_n})_{n\in\N}$. We will write
\[
\bar B_{[m,n)} = \bar B_m\cdots \bar B_{n-1} \qquad (m,n\in\N \text{ with } m\le n).
\]

Moreover, we will use the notation $\bu_n = \bar B_{{[0,n)}}^{-1} \bu$ and $\bone_n=\,^t(\bar B_{{[0,n)}})\bone$. 
Using this, we gain
$\bone_n^\bot = (^t(\bar B_{{[0,n)}})\bone)^\bot = \bar B_{{[0,n)}}^{-1}\bone^\bot$, and $\bar B_{{[m,n)}} \circ \pi_{\bu_n,\bone_n}= \pi_{\bu_m,\bone_m} \circ \bar B_{[m,n)}$. We have the following technical result. (See also \cite[Section~8]{CCL22} for related, but different results.)
\begin{lemma}\label{lem:projcahnge}
Let $\btau=(\tau_n)_{n\ge 0} \in\{\sigma_1,\sigma_2,\sigma_3,\sigma_4\}^\N$ be a directive sequence with sequence of incidence matrices $(\bar B_n)_{n\in\N}$. If $n\in\N$ satisfies $\bar B_{{[n-3,n+3)}}=(B_{\sigma_1}B_{\sigma_2}B_{\sigma_3})^2$, 
then $\| \pi_{\bu_n,\bone_n} \|_\infty \le \frac{11}7$ and $\| \pi_{\bu_n,\bone} \|_\infty \le \frac{11}7$.
\end{lemma}
\begin{proof}
Because $\bar B_{{[n-3,n+3)}}=(B_{\sigma_1}B_{\sigma_2}B_{\sigma_3})^2$ we have, by definition,  
\begin{equation}\label{eq:unbonenspan}
\begin{split}
\bu_n &\in
\,B_{\sigma_1}B_{\sigma_2}B_{\sigma_3}\mathbb{R}^3_{> 0} = 
\mathop{span}\,( ^t(4,2,1), \,^t(3,2,1),\,^t(2,1,1))_{> 0}, \\
 \bone_n &\in\,
^t(B_{\sigma_1}B_{\sigma_2}B_{\sigma_3})\mathbb{R}^3_{> 0} = 
\mathop{span}\,( ^t(4,3,2), \,^t(2,2,1),\,^t(1,1,1))_{> 0}.
\end{split}
\end{equation}
To prove the theorem we have to maximize $\| \pi_{\bu_n,\bone_n}(\bx) \|_\infty$ over all 
\begin{equation}\label{eq:conditions}
\bu_n \in
\,B_{\sigma_1}B_{\sigma_2}B_{\sigma_3}\mathbb{R}^3_{> 0},\; \bone_n \in\,
^t(B_{\sigma_1}B_{\sigma_2}B_{\sigma_3})\mathbb{R}^3_{> 0},\; \bx \text{ with }\|\bx\|_\infty=1.
\end{equation}
This is equivalent to optimizing
\begin{equation}\label{eq:homogeneity}
x_i - (x_1v_1+x_2v_2+x_3v_3)(u_1v_1+u_2v_2+u_3v_3)^{-1}u_i
\end{equation}
separately under \eqref{eq:conditions}
for each $i\in \{1,2,3\}$; 
$\bu_n=\,^t(u_1,u_2,u_3)$, $\bone_n=\,^t(v_1,v_2,v_3)$, $\bx=(x_1,x_2,x_3)$. 
Fix $i\in\{1,2,3\}$. Because the problem is linear in $\bx$ the extrema can be attained only for the vectors $\bx =(\pm1,\pm1,\pm1)$ (signs can be chosen independently). By homogeneity, in \eqref{eq:homogeneity} we may assume w.l.o.g.\ that $u_i=1$. Thus, subtracting $x_i$ in \eqref{eq:homogeneity}, it remains to optimize 
\begin{equation}\label{eq:quot1}
\frac{x_1v_1+x_2v_2+x_3v_3}{u_1v_1+u_2v_2+u_3v_3}.
\end{equation}
Because the reciprocal of \eqref{eq:quot1} is linear in $\bu_n$, for fixed $\bx,\bone_n$, the extrema are attained for $\bu_n$ in the extremal rays of the cone $B_{\sigma_1}B_{\sigma_2}B_{\sigma_3}\mathbb{R}^3_{> 0}$. After rescaling $\bu_n$, which does not change $\pi_{\bu_n,\bone_n}$, we may therefore assume that the maximum of $\| \pi_{\bu_n,\bone_n}(\bx) \|_\infty$ is attained for $\bu_n\in \{ ^t(4,2,1), \,^t(3,2,1),\,^t(2,1,1)\}$. But for each of the finitely many constellations $(\bu_n,\bx)$ we may, again by homogeneity, scale $\bone_n$ in \eqref{eq:homogeneity} in a way that $\langle \bu_n,\bone_n \rangle =1$. Thus it remains to optimize $x_1v_1+x_2v_2+x_3v_3$
which is linear in $\bone_n$, under the (also linear) condition $u_1v_1+u_2v_2+u_3v_3=1$. Thus the extremum of this linear optimization problem in $\bone_n$ is attained for 
$\bone_n$ in the extremal rays of the cone $
^t(B_{\sigma_1}B_{\sigma_2}B_{\sigma_3})\mathbb{R}^3_{> 0}$.
Rescaling $\bone_n$, we may therefore assume that the maximum of $\| \pi_{\bu_n,\bone_n}(\bx) \|_\infty$ is attained for $\bone_n \in \{ ^t(4,3,2), \,^t(2,2,1),\,^t(1,1,1)\}$. Computing $\| \pi_{\bu_n,\bone_n}(\bx) \|_\infty$ for the $2^3\cdot 3^2$ constellations of $(\bu_n,\bone_n,\bx)$,  
we get $\| \pi_{\bu_n,\bone_n}(\bx) \|_\infty \le \frac{11}7$ under the conditions \eqref{eq:conditions} and, hence, $\| \pi_{\bu_n,\bone_n} \|_\infty \le \frac{11}7$.
The estimate $\| \pi_{\bu_n,\bone}(\bx) \|_\infty \le \frac{11}7$ follows as a special case from the proof.
\end{proof}

The following lemma shows that the projections that we need to estimate in order to control balancedness according to \Cref{lem:sadiclem41} are bounded by operator norms of the matrices $\bar B_{{[m,n)}}$ when acting on certain hyperplanes.

\begin{lemma}\label{lem:threeprojOne}
Let $\btau\in\{\sigma_1,\sigma_2,\sigma_3,\sigma_4\}^\N$ be a directive sequence with generalized right eigenvector $\bu$ and with sequence of incidence matrices $(\bar B_n)_{n\in\N}$.
 Let $\bv\in \mathbb{R}_{\ge 0}^3 \setminus\{\mathbf{0}_3\}$ and $m,n\in\N$ with $m<n$ be given in a way that $\bar B_{{[n-3,n+3)}}=(B_{\sigma_1}B_{\sigma_2}B_{\sigma_3})^2$. If $\Vert \pi_{\bu_n,\bone}(\bv) \Vert_\infty \le C$ holds for some $C>0$, then $\Vert \pi_{\bu_{m},\bone}( \bar B_{{[m,n)}}\bv) \Vert_\infty \le \frac{22}7 C \Vert \bar B_{{[m,n)}}|_{\bone_{n}^\bot} \Vert_\infty$.
\end{lemma}

\begin{proof}
Represent $\bv = \bz + \alpha\bu_n$ with $\bz\in \bone_n^\bot$. Then $\bar B_{{[m,n)}}\bv = \bar B_{{[m,n)}}\bz + \alpha\bu_{m}$ with $\bar B_{{[m,n)}}\bz\in \bone_m^\bot$, and we get $\Vert \pi_{\bu_{n},\bone}(\bz) \Vert_\infty=\Vert \pi_{\bu_{n},\bone}(\bv) \Vert_\infty$ and $\Vert \pi_{\bu_{m},\bone}(\bar B_{{[m,n)}}\bz) \Vert_\infty=\Vert \pi_{\bu_{m},\bone}(\bar B_{{[m,n)}}\bv) \Vert_\infty$. Thus we may assume w.l.o.g.\ that $\bv \in \bone_n^\bot$. 
By assumption and using \Cref{lem:projcahnge} we gain
\[
\Vert \bv\Vert_\infty = \Vert \pi_{\bu_{n},\bone_n}(\bv) \Vert_\infty  \le \frac{11}7\Vert \pi_{\bu_{n},\bone}(\bv) \Vert_\infty \le \frac{11}7 C.
\]
Because $\Vert \pi_{\bw,\bone} \Vert_\infty \le 2$ holds for each 
$\bw\in  \mathbb{R}_{\ge 0}^3 \setminus\{\mathbf{0}_3\}$ this  yields
\[
\begin{split}
\Vert \pi_{\bu_{m},\bone}(\bar B_{{[m,n)}}\bv) \Vert_\infty 
\le 2 \Vert \bar B_{{[m,n)}}|_{\bone_{n}^\bot} \Vert_\infty \cdot \Vert \bv\Vert_\infty 
\le \frac{22}7 C \Vert \bar B_{{[m,n)}}|_{\bone_{n}^\bot} \Vert_\infty, 
\end{split}
\]
and the lemma is proved. 
\end{proof}

According to the previous lemma, in order to control balancedness, we need to bound the operator norm of the products $\bar B_{{[m,n)}}$ on certain hyperplanes. This is done in the subsequent lemmas. The block $(\sigma_1,\sigma_2,\sigma_3)$ occurring in the statement of \Cref{th:balance} plays a decisive role here. Indeed, as seen in the next lemma, for this block the matrices $\bar B_{{[m,n)}}$ have particularly good contraction properties on the relevant hyperplanes.

\begin{lemma}[{{\em cf.}~\cite[Lemma~6]{Delecroix-Hejda-Steiner}}]
\label{lem:DHSLem6}
Let $\btau\in\{\sigma_1,\sigma_2,\sigma_3,\sigma_4\}^\N$ be a sequence of substitutions 
with sequence of incidence matrices $(\bar B_n)_{n\in\N}$. If $\bar B_{{[n-3,n+3)}} = (B_{\sigma_1}B_{\sigma_2}B_{\sigma_3})^2$ we have 
\begin{equation}\label{eq:57}
\Vert (\bar B_{{[n,n+3)}})|_{\bone_{n+3}^\bot} \Vert_\infty \le \frac{5}{7}.
\end{equation}
\end{lemma}

For all the other blocks we just need to make sure that the matrices $\bar B_{{[m,n)}}$ do not expand too much. This is established now. We start with two preparatory lemmas.

\begin{lemma} \label{lem:word}
Let $\bx \in \N^3$ be given. 
Then there exists a word $w \in \cA^*$ such that $\mathbf{l}(w) = \bx$ and  $\|\pi_{\bx,\bone} \mathbf{l}(p)\|_\infty \le 1$ for each prefix $p$ of $w$. 
\end{lemma}

\begin{proof}
This is a special case of a result proved in \cite{Meijer:73,Tijdeman:80}.
\end{proof}

Let $X,Y\subset \mathbb{R}^3$. We say that $X$ is relatively dense in $Y$ with denseness constant $K>0$, if $Y$ can be covered by balls (w.r.t.\ the $\|\cdot\|_\infty$-norm) of radius $K$ centered at $X$.

\begin{lemma}\label{lem:claim1}
Let $\btau\in\{\sigma_1,\sigma_2,\sigma_3,\sigma_4\}^\N$ be a sequence of substitutions with generalized right eigenvector~$\bu$ and let $C>0$. If $n\in\N$ satisfies $\bar B_{{[n-3,n+3)}} = (B_{\sigma_1}B_{\sigma_2}B_{\sigma_3})^2$, then there exists a $\frac{44}7C$-letter balanced word $v=v_0v_1\cdots \in\cA^\N$ with frequency vector $\bu_n$ such that the set $\{\pi_{\bu_n,\bone_n}(\mathbf{l}(v_0\cdots v_\ell)) : \ell \ge 0\}$ is relatively dense in 
$\{\bz\in \bone_n^\bot\;:\; \|\bz\|_\infty \le C\}$,
with a denseness constant $K$ not depending on $C$. 
\end{lemma}

\begin{proof}
By Lemma~\ref{lem:word}, for $\bx,\by\in \N^3$ with $\by-\bx\in \N^3$ there is a word $w\in\cA^*$ such that $\mathbf{l}(w) = \by-\bx$ and  $\|\pi_{\by-\bx,\bone_n} \mathbf{l}(p)\|_\infty \le \|\pi_{\by-\bx,\bone_n}\|_\infty$ for each prefix $p$ of $w$. 
From the definition of $\pi_{\by-\bx,\bone_n}$ in \eqref{eq:projections} we gain
\[
\bx+\mathbf{l}(p)=
\bx + \frac{\langle \mathbf{l}(p),\bone_n\rangle}{\langle \by-\bx,\bone_n\rangle}(\by-\bx) + \pi_{\by-\bx,\bone_n}(\mathbf{l}(p)).
\]
Because
$0\le \langle \mathbf{l}(p),\bone_n\rangle \le \langle \by-\bx,\bone_n\rangle$
this implies that
\begin{equation}\label{eq:narrowBL}
\begin{split}
\|\pi_{\bu_n,\bone_n}(\bx+\mathbf{l}(p))\|_\infty &\le
\|\max\{\pi_{\bu_n,\bone_n}(\bx + \lambda(\by-\bx))\;:\; 0\le \lambda\le 1\}\|_\infty\\
&\qquad\qquad\qquad\qquad\qquad\qquad
+\| \pi_{\bu_n,\bone_n} \circ \pi_{\by-\bx,\bone_n}(\mathbf{l}(p))\|_\infty
\\
&\le \max\{\|\pi_{\bu_n,\bone_n}(\bx)\|_\infty, \|\pi_{\bu_n,\bone_n}(\by)\|_\infty\} + \|\pi_{\by-\bx,\bone_n}\|_\infty
\end{split}
\end{equation}
(note that $\pi_{\bu_n,\bone_n}$ is the identity on $\bone_n^\bot$).

Define 
$
S=\{\bx \in \N^3 \,:\, \|\pi_{\bu_n,\bone_n}(\bx)\|_\infty < C-2\|\pi_{\bu_n,\bone_n}\|_\infty \}
$.
Because $S$ is relatively dense in the cylinder $Z=\{\bx \in\mathbb{R}^3_{\ge 0} \,:\, \|\pi_{\bu_n,\bone_n}(\bx)\|_\infty < C \}$ with a denseness constant not depending on $C$, the projection 
$\pi_{\bu_n,\bone_n}(S \cap \{\bx: \langle \bx, \bu_n\rangle > N\})$ is relatively dense in 
$\pi_{\bu_n,\bone_n}(Z)=\{z \in \bone_n^\bot \,:\, \|z\|_\infty < C \}$ for each $N\in \N$ with a denseness constant not depending on $N$ and $C$.
Thus there is sequence $(\bx_i)_{i\in \N}\subset S$ with $\bx_0=0$ such that $\pi_{\bu_n,\bone_n}((\bx_i)_{i\in \N})$ is relatively dense in $\{\bz \in \bone_n^\bot \,:\, \|z\|_\infty < C \}$ with a denseness constant $K$ not depending on $C$ and $\bx_i-\bx_{i-1} \in \N^3$ points in a direction close to $\bu_n$ in the sense that $\|\pi_{\bx_i-\bx_{i-1},\bone_n}\|_\infty < 2\|\pi_{\bu_n,\bone_n}\|_\infty$ ($i\ge 1$)\footnote{The properties of $\bx_i-\bx_{i-1}$ can be achieved by choosing the  vectors $\bx_i$ far apart from each other in the cylinder $Z$, {\it i.e.},  $\bx_i$ lies much ``higher'' than $\bx_{i-1}$ in this cylinder (w.r.t.\ the ``height'' $\langle \bx, \bu_n\rangle$).}. By \eqref{eq:narrowBL} and the properties of $(\bx_i)_{i\in \N}$ for each $i\ge 1$ there exists a word $w_i\in \cA^*$ with $\mathbf{l}(w_i)=\bx_i-\bx_{i-1}$ such that 
\[
\|\pi_{\bu_n,\bone_n}(\bx_{i-1}+\mathbf{l}(p))\|_\infty \le \max\{\|\pi_{\bu_n,\bone_n}(\bx_{i-1})\|_\infty, \|\pi_{\bu_n,\bone_n}(\bx_i)\|_\infty\} + \|\pi_{\bx_i-\bx_{i-1},\bone_n}\|_\infty \le C
\]
for each prefix $p$ of $w_i$.
Thus the concatenation $v=w_1w_2\cdots\in\cA^\N$ satisfies
\begin{equation}\label{eq:xipC}
\{\pi_{\bu_n,\bone_n}(\bx_i)\;:\; i\in \N\} \subset
\{\pi_{\bu_n,\bone_n}(\mathbf{l}(p))\;:\; p \text{ prefix of } v \} \subset \{z \in \bone_n^\bot \;:\; \|z\|_\infty < C \}, 
\end{equation}
and, hence,
$\{\pi_{\bu_n,\bone_n}(\mathbf{l}(p))\,:\, p \text{ prefix of } v \}$ is relatively dense in $\{z \in \bone_n^\bot \,:\, \|z\|_\infty < C \}$ with denseness constant $K$. Moreover, using \Cref{lem:projcahnge} we gain from \eqref{eq:xipC} that
\[
\bigg\{\pi_{\bu_n,\bone}(\mathbf{l}(p))\;:\; p \text{ prefix of } v \} \subset \{z \in \bone^\bot \;:\; \|z\|_\infty < \frac{11}7C \bigg\}.
\]
Thus, by \Cref{lem:sadiclem41}~(ii), $v$ is $\frac{44}{7}C$-balanced with frequency vector~$\bu_n$.
\end{proof}

Now we state the desired bound for the operator norms.

\begin{lemma}\label{lem:Mbound}
Let $\btau\in\{\sigma_1,\sigma_2,\sigma_3,\sigma_4\}^\N$ be a sequence of substitutions with generalized right eigenvector $\bu$ and with sequence of incidence matrices $(\bar B_n)_{n\in\N}$.
For each $0\le m\le n$ with $m=0$ or $\bar B_{{[m-3,m+3)}}=(B_{\sigma_1}B_{\sigma_2}B_{\sigma_3})^2$ we have
\[
\Vert (\bar B_{{[m,n)}})|_{\bone_{n}^\bot} \Vert_\infty \le 10. 
\]
\end{lemma}

\begin{proof}
Fix $m,n$ with $0\le m<n$ and let $C\in\mathbb{N}$ be arbitrary. 
Let $v\in \cA^\N$ be as in \Cref{lem:claim1}.
Then $\sigma_{[m,n)}(v)$ has frequency vector $\bu_m$ and is $(\frac{44}{7}C+4(m-n))$-letter balanced by \Cref{lem:sigma_ibalance}. By \Cref{lem:projcahnge} and \Cref{lem:sadiclem41}~(ii) the images $B_{[m,n)} \pi_{\bu_n,\bone_n}(\mathbf{l}(v_0\cdots v_k))$ satisfy 
\[
\begin{split}
\Vert \bar B_{{[m,n)}} \pi_{\bu_n,\bone_n}(\mathbf{l}(v_0\cdots v_k)) \Vert_\infty&=\Vert \pi_{\bu_m,\bone_m}(\mathbf{l}(\tau_{[m,n)}(v_0\cdots v_k))) \Vert_\infty
\\&\le \frac{11}{7}\cdot \Vert \pi_{\bu_m,\bone}(\mathbf{l}(\tau_{[m,n)}(v_0\cdots v_k))) \Vert_\infty
\le 10(C+(n-m))
 \end{split}
\]
(for the case $m=0$ the constant $\frac{11}7$ can even be replaced by $1$ because $\bone_0=\bone$).
By the properties of $v$ asserted in \Cref{lem:claim1} there is a $K>0$ independent of $C$ such that for each $\bx\in \bone_n^\bot$ with $\|\bx\|_\infty \le C$ there is an $\ell\ge 0$ with $\|\bx-\pi_{\bu_n,\bone_n}(\mathbf{l}(v_0\cdots v_\ell))\|_\infty \le K$. Thus,
\[
\begin{split}
 \|\bar B_{{[m,n)}} \bx\|_\infty & \le
 \|\bar B_{{[m,n)}}(\bx-\pi_{\bu_n,\bone_n}(\mathbf{l}(v_0\cdots v_\ell)))\|_\infty + \|\bar B_{{[m,n)}}(\pi_{\bu_n,\bone_n}(\mathbf{l}(v_0\cdots v_\ell)))\|_\infty \\
 & \le K \Vert \bar B_{{[m,n)}}|_{\bone_{n}^\bot}\Vert_\infty
 + 10(C+(n-m)) \\
 &= 10C\bigg(1 + \frac{K \Vert \bar B_{{[m,n)}}|_{\bone_{n}^\bot}\Vert_\infty+10(n-m)}{10C}\bigg).
\end{split}
\]
Because $\Vert \bar B_{{[m,n)}}|_{\bone_{n}^\bot}\Vert_\infty$ is finite and does not depend on $C$, this implies that for each $\varepsilon >0$ there is a $C>0$ such that each $\bx\in \bone_n^\bot$ with $\|\bx\|_\infty=C$ satisfies $\|\bar B_{{[m,n)}} \bx\|_\infty <10C+\varepsilon$. 
Since $C$ can be arbitrarily large this yields $\Vert \bar B_{{[m,n)}}|_{\bone_{n}^\bot}\Vert_\infty \le 10$ by the definition of the operator norm.
\end{proof}

We are now ready to prove letter balancedness for the relevant class of languages related to the Reverse algorithm. The following proposition concludes the main step of the proof of Theorem~\ref{th:balance}.

\begin{proposition}\label{th:balanceLetter}
If $\btau \in \{\sigma_1,\sigma_2,\sigma_3,\sigma_4\}^\N$ is a sequence of substitutions in which the word $({\sigma_1}{\sigma_2}{\sigma_3})^{9}$ has positive density, then $\cL_{\btau}$ is letter balanced.
\end{proposition}

\begin{proof}
Let $\btau=(\tau_n)_{n\in\N}$ be as in the statement and let $(\bar B_n)_{n\in\N}$ be its sequence of incidence matrices. Because the matrix $B=B_{\sigma_1}B_{\sigma_2}B_{\sigma_3}$ is positive,
\Cref{lem:3.5.5} implies that $\btau$ has a generalized right eigenvector $\bu$. Therefore, in view of Lemma~\ref{lem:sadiclem41}~(i), it suffices to show that there exists some $C\in\mathbb{N}$ such that for each $u\in\mathcal{L}_{\btau}$ we have 
$\Vert \pi_{\bu,\bone}(\mathbf{l}(u)) \Vert_\infty \le C$. 

There is a strictly increasing sequence $(\ell_k)_{k\in \N}$ satisfying $\bar B_{{[\ell_k-3, \ell_k+24)}} = B^{9}$,  with $\ell_k\ge 3$ as small as possible $(k\ge 0)$. Put $\ell_{-1}=0$ and set $g_k=\ell_k-\ell_{k-1}$ for $k\ge 0$. 
Let now $u\in\mathcal{L}_{\btau}$ arbitrary. Then for some $n\in\mathbb{N}$ we may write
\begin{equation}\label{eq:urep}
u =  s_0\tau_{[0,\ell_0)}(s_{1})\cdots\tau_{[0,\ell_{n{-}2})}(s_{n-1})\tau_{[0,\ell_{n-1})}(u_{n})\tau_{[0,\ell_{n{-}2})}(p_{n-1}) \cdots  \tau_{[0,\ell_0)}(p_{1}) p_0,
\end{equation}
where $s_k,p_k\prec\tau_{[\ell_{k-1},\ell_k)}(i_k)$ for some $i_k\in \cA$ ($0\le k< n$) and $u_n\prec\tau_{[\ell_{n-1},\ell_n)}(i_n)$ for $i_n\in \cA$. 

Set $m=\max\{\|B_{\sigma_j}\|_\infty\,:\, 1\le j\le 3\}$ and 
consider $\tau_{[0,\ell_{k-1})}(v)$ with $k\in\N$ and $v\prec \tau_{[\ell_{k-1},\ell_k)}(i)$ for some $i\in \cA$. 
Because $\Vert \pi_{\bw,\bone} \Vert_\infty \le 2$ holds for each 
$\bw\in  \mathbb{R}_{\ge 0}^3 \setminus\{\mathbf{0}_3\}$ we get 
\[
\Vert \pi_{\bu_{\ell_{k-1}},\bone}(\mathbf{l}(v))\Vert_\infty  \le 2\Vert\mathbf{l}(v)\Vert_\infty 
\le 2\Vert\mathbf{l}(\tau_{[\ell_{k-1},\ell_k)}(i))\Vert_\infty
= 2\Vert \bar B_{{[\ell_{k-1},\ell_k)}}\mathbf{l}(i)\Vert_\infty
\le 2m^{g_k}. 
\]
From \Cref{lem:threeprojOne} we therefore gain that 
\begin{equation}\label{eq:MMtil}
\Vert \pi_{\bu,\bone}( \bar B_{{[0,\ell_{k-1})}}\mathbf{l}(v)) \Vert_\infty \le \frac{44}{7}m^{g_k} \Vert \bar B_{{[0,\ell_{k-1})}}|_{\bone_{\ell_{k-1}}^\bot} \Vert_\infty< 7m^{g_k} \Vert \bar B_{{[0,\ell_{k-1})}}|_{\bone_{\ell_{k-1}}^\bot} \Vert_\infty.
\end{equation}
By construction, 
\[
\bar B_{{[0,\ell_{k-1})}}=  \bar B_{{[0,\ell_0)}} B^7  \bar B_{{[\ell_0+21,\ell_1)}} B^7 \cdots 
 \bar B_{{[\ell_{k-3}+21,\ell_{k-2})}}  B^7
 \bar B_{{[\ell_{k-2}+21,\ell_{k-1})}}.
\]
By the definition of the sequence $(\ell_k)_{k\in\N}$, we may apply \Cref{lem:DHSLem6} to $\Vert B^7|_{\bone_{\ell_{i-1}+21}^\bot} \Vert_\infty$ as well as \Cref{lem:Mbound} to $\Vert  \bar B_{{[0,\ell_1)}}|_{\bone_{\ell_{1}}^\bot} \Vert_\infty$ and $\Vert  \bar B_{{[\ell_{i-1}+21,\ell_i)}}|_{\bone_{\ell_{i}}^\bot} \Vert_\infty$ ($1\le i\le k-1$). Setting $q=10\cdot\left(\frac{5}{7}\right)^7<1$, we therefore gain from \eqref{eq:MMtil} that
\begin{equation}\label{eq:pivest}
\begin{split}
&\Vert \pi_{\bu,\bone}( \bar B_{{[0,\ell_{k-1})}}\mathbf{l}(v)) \Vert_\infty < 7m^{g_k} \Vert \bar B_{{[0,\ell_{k-1})}}|_{\bone_{\ell_{k-1}}^\bot} \Vert_\infty 
\\
&\qquad\le 7m^{g_k} \Vert   \bar B_{{[0,\ell_0)}}|_{\bone_{\ell_{0}}^\bot} \Vert_\infty \cdot   \prod_{i=1}^{k-1} \Vert \big(B^7|_{\bone_{\ell_{i-1}+21}^\bot} \Vert_\infty  \cdot \Vert  \bar B_{{[\ell_{i-1}+21,\ell_i)}}|_{\bone_{\ell_{i}}^\bot}\Vert_\infty\big) 
\le 70 m^{g_k} q^{k-1}
\end{split}
\end{equation}
(note that this is valid also for $k=0$). Combining \eqref{eq:pivest} with \eqref{eq:urep} we see that
\begin{equation}\label{eq:finEst}
\begin{split}
 \Vert \pi_{\bu,\bone}(\mathbf{l}(u)) \Vert_\infty 
\le& 
\sum_{k=0}^{n-1}
\big(
\Vert \pi_{\bu,\bone}( \bar B_{{[0,\ell_{k-1})}}\mathbf{l}(s_k)) \Vert_\infty+
\Vert \pi_{\bu,\bone}( \bar B_{{[0,\ell_{k-1})}}\mathbf{l}(p_k)) \Vert_\infty
\big)\\
&+\Vert \pi_{\bu,\bone}( \bar B_{{[0,\ell_{n-1})}}\mathbf{l}(u_n)) \Vert_\infty 
\le 140 \sum_{k=0}^{n} m^{g_k} q^{k-1}.
\end{split}
\end{equation}
Because  
    $\lim_{n\rightarrow \infty}\frac{|\tau_1\tau_2\cdots\tau_n|_{(\sigma_1\sigma_2\sigma_3)^{9}}}{n}>0$,
we have 
$\lim_{k\rightarrow\infty}\frac{\ell_k}{k}=
\lim_{n\rightarrow\infty}\frac{g_1+g_2+\cdots+g_k}{k}= g.
$
Therefore, for each $\varepsilon>0$, there exists $k_0$ such that 
$   \frac{g_1+g_2+\cdots+g_k}{k}\in \big(g-\frac\varepsilon2, g+\frac\varepsilon2\big)
$
and, hence, $g_k<\varepsilon k$ for each $k\ge k_0$. This implies that 
\[
\sum_{k\ge 0} m^{g_k} q^{k} <
\sum_{k= 0}^{k_0-1} m^{g_k} q^{k} + \sum_{k\ge k_0} m^{\varepsilon k} q^{k}.
\]
Choosing $\varepsilon$ in a way that $m^\varepsilon q <1$, this series converges. Therefore, \eqref{eq:finEst} implies that there is $C>0$ not depending on $u \in \mathcal{L}_{\btau}$ such that $\Vert \pi_{\bu,\bone}(\mathbf{l}(u)) \Vert_\infty <C $. This finishes the proof.
\end{proof}

The proof of our main result is now a matter of few lines.

\begin{proof}[Proof of Theorem $\ref{th:balance}$]
By \Cref{th:balanceLetter}, the language $\cL_{\btau}^{(k)}$ is letter balanced for infinitely many~$k$. 
To prove factor balancedness, we want to apply \cite[Theorem 4.1]{PSt:24}. This result is valid only for sequences of proper substitutions.
Because $\sigma_4$ is not proper, we need to consider compositions of the form $\sigma_4^k\sigma_i$, which are proper for each $i\in\{1,2,3\}$ and each $k \in \mathbb{N}$. By blocking consecutive substitutions, we see that for each $\btau\in\{\sigma_1, \sigma_2, \sigma_3, \sigma_4\}^{\mathbb{N}}$, there exists $\btau'\in\{\sigma_4^k\sigma_1, \sigma_4^k\sigma_2, \sigma_4^k\sigma_3\,:\, k\in\N\}^{\mathbb{N}}$ such that $\cL_{\btau}=\cL_{\btau'}$.
Because $\cL_{\btau}=\cL_{\btau'}$ and $\btau'$ is a sequence of proper substitutions,~\cite[Theorem 4.1]{PSt:24} yields that $\cL_{\btau'}=\cL_{\btau}$ is even factor balanced. 
\end{proof}

\appendix
\section{The sorted Reverse algorithm and its second Lyapunov exponent}
\label{sec:appA}
We come back to the sorted version of the Reverse algorithm $\mathop{sort}\circ F_{R}$ mentioned in \Cref{rem:sorted}. This algorithm is defined on the domain $\Lambda'=\{[x_0:x_1:x_2]\in\mathbb{P}^2 \,:\, x_0>x_1>x_2>0\}$. By projecting it via $p:\Lambda'\to\Delta'$; $[x_0:x_1:x_2]\mapsto \big(\frac{x_1}{x_0},\frac{x_2}{x_0}\big)$, we obtain the sorted version $(\Delta', f'_R)$, where $\Delta'=\{(x_1,x_2)\in\mathbb{R}^2 \,:\, 1>x_1>x_2>0\}$. Let $\Lambda_{\pi}=\{[x_0:x_1:x_2]\in\mathbb{P}^2 \,:\, x_{\pi_0}>x_{\pi_1}>x_{\pi_2}>0\}$, where $\pi$ is a permutation and let $\Delta'((i,\pi))=p(\Lambda'\cap\Lambda_i\cap F_R^{-1}(\Lambda_{\pi}))$. Set
\begin{align*}
A':\Delta' \to \mathcal{M}(3,\mathbb{N});\quad  \bx \mapsto \,^t\!M_{i,\pi}  \quad \text{if and only if} \quad \bx \in\Delta'((i,\pi)),
\end{align*}
for $(i,\pi)\in \mathcal{A}'=\{(1,id),(1,(213)),(1,(231)), (4,(321))\}$, where 
\begin{align*}\label{eq:matricesM(Sorted)}
M_{1,id}=
\begin{pmatrix}
1 & 1 & 1\\
0 & 1 & 0\\
0 & 0 & 1
\end{pmatrix}, \,
M_{1,(213)}=
\begin{pmatrix}
1 & 1 & 1\\
1 & 0 & 0\\
0 & 0 & 1
\end{pmatrix}, \,
M_{1,(231)}=
\begin{pmatrix}
1 & 1 & 1\\
1 & 0 & 0\\
0 & 1 & 0
\end{pmatrix}, \,
M_{4,(321)}=
\begin{pmatrix}
1 & 1 & 0\\
1 & 0 & 1\\
0& 1 & 1
\end{pmatrix}.
\end{align*}

The map $f'_R$ is defined by the map
\begin{align*}
 f'_{\mathrm{R}}: \Delta'\to\Delta';\qquad \bx \mapsto \frac{\,^t\!A'(\bx)^{-1}\bx}{\|\,^t\!A'(\bx)^{-1}\bx\|_1}
\end{align*}
and let $F'_R=\,^t\!A'(\bx)^{-1}\bx$.

Let
$
\Delta'^{(\#)}=
\{(y_1,y_2)\in\mathbb{R}^2_{+} \;:\; y_1+y_2\ge1,~ |y_1-y_2|\le1\}.
$
We define the dual map $f'^{(\#)}_R$ on $\Delta'^{(\#)}$ by 
$\by\mapsto \frac{A'(\by)\cdot \by}{\|A'(\by)\cdot \by\|_1}$.
By \cite[Theorem~13 in Chapter~3]{Schweiger:00}, we obtain the density function
\begin{align*}
    \int_{\Delta'^{(\#)}}\frac{1}{(1+x_1y_1+x_2y_2)^{3}}dy_1dy_2=\frac{1}{(1+x_1)(1+x_2)(x_1+x_2)}
\end{align*}
of an absolutely continuous invariant measure $\mu'$. In the same way as in \Cref{sec:erg}, we can see that the sorted version is ergodic with respect to $\mu'$. Also, we can see $F'_R$ and $F_R$ have the same Lyapunov exponents. Indeed, we have $F'^{n}_R\circ \mathop{sort}=\mathop{sort}\circ (F_R^{n}\circ \mathop{sort})^n$, this implies $\lambda_i(F'_R)=\lambda_i(F_R\circ \mathop{sort})$ for $i\in\{1,2,3\}$. The map $F_R\circ \mathop{sort}$ is defined on $\Lambda\setminus\Lambda_{(132)}\cup\Lambda_{(231)}$, and then $\lambda_i(F_R\circ \mathop{sort})=\lambda_i(F_R)$ by symmetry. Set
\begin{align*}
D'^{(n)}(\bx)=\begin{pmatrix}
0 & 1 & 0\\
0 & 0 & 1
\end{pmatrix} A'^{(n)}(\bx)
\begin{pmatrix}
-{x_1} & -{x_2}\\
1 & 0\\
0 & 1
\end{pmatrix}.
\end{align*}
Then, this also has the cocycle property, and we can estimate the second Lyapunov exponent of $A'^{(n)}$ in the same way in \Cref{sec:lyap}. We define $\mathbb{L}'_1(n)$ and $\mathbb{L}'_2(n)$ as follows ($\Delta'(w)$, $w\in \cA'^*$, is defined in analogy to the cylinders $\Delta(w)$ for $w\in\cA^*$, see \eqref{eq:DeltaCyl}).
\begin{align*}
\mathbb{L}'_1(n)
=&\frac{24}{\pi^2n}\max_{\bx\in\Delta'((1,~id)^n)}\log{\|D'^{(n)}(\bx)\|_1}\int_{\Delta'((1,~id)^n)}d\mu
=\frac{24}{\pi^2n}\log{\Big(1+\frac{n}{n+1}\Big)}\int_{\Delta'((1,~id)^n)}d\mu'
\end{align*}
and
\begin{align*}
\mathbb{L}'_2(n)=&\frac{24}{\pi^2n}\sum_{\substack{w\in\mathcal{A}'^n\setminus\{(1,id)\}^n,\\ \max\{\log{\|D'^{(n)}(\bx)\|_1}\,:\,\bx\in\Delta'(w) \}>0}}
\!\!\!\!\!\!\!\!\bigg(
\max_{\bx\in \Delta'(w)}\frac{1}{(1+x_1)}\cdot
\max_{\bx\in \Delta'(w)}\frac{1}{(1+x_2)}\cdot
\max_{\bx\in \Delta'(w)}\frac{1}{(x_1+x_2)}\bigg)
\\
&\hskip5cm\cdot 
\mathrm{Leb}(\Delta'(w))\cdot
\max_{\bx\in\Delta'(w)}\log{\|D'^{(n)}(\bx)\|_1}\\
+&\frac{24}{\pi^2n}\sum_{\substack{w\in\mathcal{A}'^n\setminus\{(1,id)\}^n, \\ \max\{\log{\|D'^{(n)}(\bx)\|_1}\,:\,\bx\in\Delta'(w) \}<0}}
\!\!\!\!\!\!\!\!\bigg(
\min_{\bx\in \Delta'(w)}\frac{1}{(1+x_1)}\cdot
\min_{\bx\in \Delta'(w)}\frac{1}{(1+x_2)}\cdot
\min_{\bx\in \Delta'(w)}\frac{1}{(x_1+x_2)}\bigg)
\\
&\hskip5.2cm\cdot 
\mathrm{Leb}(\Delta'(w)) \cdot
\max_{\bx\in\Delta'(w)}\log{\|D'^{(n)}(\bx)\|_1}.
\end{align*}

Then, by computer calculations, the value of $\mathbb{L}'_1(n)+\mathbb{L}'_2(n)$ is negative for the first time when $n=11$; more precisely, we find
$\mathbb{L}'_1(11)+\mathbb{L}'_2(11) < 0.005873-0.008701=-0.002828$.

\section*{Acknowledgements}
We wish to thank Val\'erie Berth\'e, Wolfgang Steiner, and the anonymous referee for valuable comments that helped to improve this paper.

\bibliographystyle{siam}  
\bibliography{lit}
\end{document}